\documentclass[10 pt, twoside, reqno]{amsproc}
\usepackage{amssymb,latexsym,amsthm,amsmath,mathtools,amsfonts,hyperref,fancyhdr,comment,enumitem, cleveref,xcolor,bbm, ytableau, caption, amssymb, amsaddr}

\hypersetup{
	colorlinks=true,
	linkcolor=blue,
	filecolor=magenta,
	urlcolor=cyan,
	citecolor=red,
}

\newcommand{\psl}{\textup{PSL}}

\newcommand{\tensor}{\otimes}
\newcommand{\Mod}{\mathrm{mod}}

\newcommand{\ind}{\mathrm{Ind}}
\newcommand{\res}{\mathrm{Res}}
\newcommand{\irr}{\mathrm{Irr}}
\newcommand{\ccn}{\mathrm{ccn}}
\newcommand{\ch}{\mathrm{ch}}

\newtheorem{theorem}{Theorem}[section]
\newtheorem{lemma}[theorem]{Lemma}
\newtheorem{corollary}[theorem]{Corollary}
\newtheorem{proposition}[theorem]{Proposition}
\newtheorem*{theorem*}{Theorem}
\newtheorem{definition}[theorem]{Definition}

\newtheorem{remark}[theorem]{Remark}

\newtheorem{example}[theorem]{Example}
\newtheorem{conjecture}[theorem]{Conjecture}
\numberwithin{equation}{section}
\newcommand{\ignore}[1]{}

\newcommand{\mynote}[1]{}

\begin{document}
	
	\title{Covering Numbers of Some Irreducible Characters of the Symmetric Group}
	\email{rijubrata8@gmail.com, velmurugan@imsc.res.in}
	\author[Rijubrata Kundu and Velmurugan S]{Rijubrata Kundu and Velmurugan S}\address{The Institute of Mathematical Sciences, Chennai}\address{Homi Bhabha National Institute, Mumbai}
	
	\date{\today}
	
	\begin{abstract}
		The covering number of a non-linear character $\chi$ of a finite group $G$ is the least positive integer $k$ such that every irreducible character of $G$ occurs in $\chi^k$. We determine the covering numbers of irreducible characters of the symmetric group $S_n$ indexed by certain two-row partitions (and their conjugates), namely $(n-2,2)$ and $((n+1)/2, (n-1)/2)$ when $n$ is odd. We also determine the covering numbers of irreducible characters indexed by certain hook-partitions (and their conjugates), namely $(n-2,1^2)$, the almost self-conjugate hooks $(n/2+1, 1^{n/2-1})$ when $n$ is even, and the self-conjugate hooks $((n+1)/2, 1^{(n-1)/2})$ when $n$ is odd.
 	\end{abstract}

	\maketitle
	
	\section{Introduction}
	
	Arad, Chillag, and Herzog (see \cite{ach}) introduced the notion of covering number of a character of a finite group analogous to the notion of covering number of conjugacy classes of groups (see \cite{ash}). Let $G$ be a finite group, and $\irr(G)$ denote the set of all irreducible characters of $G$.  Let $\irr(G)^{+}:=\{\chi \in \irr(G) \mid \chi(1)>1\}$. For characters  $\chi$ and $\rho$ of $G$, recall that the product $\chi\rho$ is the character of $G$ for the internal tensor product of the respective representations of $G$ that $\chi$ and $\rho$ afford.  Let $c(\chi)$ denote the set of all the irreducible constituents of a character $\chi$ of $G$. The covering number of $\chi$, denoted by $\ccn(\chi;G)$, is the least positive integer $k$ (if it exists) such that $c(\chi^k)=\irr(G)$. Suppose that $\ccn(\chi;G)$ exists for all $\chi \in \irr(G)^{+}$. Then, the character-covering-number of $G$, denoted by $\ccn(G)$, is the least positive integer $m$ such that $c(\chi^m)=\irr(G)$ for all $\chi\in \irr(G)^{+}$. In other words, $\ccn(G):=\text{max}\{\ccn(\chi;G) \mid \chi \in \irr(G)^{+}\}$.
	
	\medskip
	
	In \cite[Theorem 1]{ach}, the authors proved that if $G$ is a finite non-abelian simple group, then $\ccn(\chi;G)$ exists for all $\chi \in \irr(G)^{+}$, and hence $\ccn(G)$ exists. Among several other results, assuming $G$ is a finite non-abelian simple group, they gave bounds for $\ccn(G)$ in terms of the number of conjugacy classes of $G$. Zisser (see \cite{zi}) proved that $\ccn(A_5)=3$ and $\ccn(A_n)=n-\lceil \sqrt{n} \;\rceil$ when $n\geq 6$. Very recently, Miller (see \cite{mi}) has shown that $\ccn(S_n)$ exists when $n\geq 5$ and that $\ccn(S_n)=n-1$. The irreducible characters of $S_n$ are parameterized by partitions of $n$. For a partition $\lambda\vdash n$, let $\chi_{\lambda}$ denote the irreducible character of $S_n$ indexed by $\lambda$. It is not difficult to see that $\ccn(\chi_{(n-1,1)};S_n)=\ccn(\chi_{(2,1^{n-2})};S_n)=n-1$. It is natural to ask whether there are other irreducible characters of $S_n$ whose covering number is $n-1$. The first result of this article shows that this is not possible.
	
	\begin{theorem}\label{Theorem_1}
		Let $n\geq 5$ and $\lambda$ be a partition of $n$. Assume that $\lambda \notin \{(n), (1^n), (n-1,1), (2,1^{n-2})\}$. Then  $\ccn(\chi_{\lambda};S_n)\leq \left \lceil \frac{2(n-1)}{3} \right\rceil$. Moreover, $\ccn(\chi_{(n-2,2)}; S_n)=\ccn(\chi_{(2^2,1^{n-4})};S_n)= \left\lceil \frac{2(n-1)}{3} \right\rceil$.
	\end{theorem}

	We determine the covering number of some other irreducible characters of $S_n$.  We have the following two theorems:
	
	\begin{theorem}\label{Theorem_2}
		Let $n\geq 5$ be odd and $k=\frac{n+1}{2}$. Then $\ccn(\chi_{(k,k-1)};S_n)=\ccn(\chi_{(2^{k-1},1)};S_n)=\lceil\log_2 n \rceil$. 
	\end{theorem}
	
	\begin{theorem}\label{Theorem_3}
		Let $n\geq 5$.
		\begin{enumerate}
			\item $\ccn(\chi_{(n-2,1^2)};S_n)=\ccn(\chi_{(3,1^{n-3})};S_n)=\lfloor \frac{n}{2} \rfloor$.
			
			\item Let $\lambda=(\frac{n+1}{2},1^{\frac{n-1}{2}})$ when $n$ is odd and $\lambda=(\frac{n}{2}+1, 1^{\frac{n}{2}-1})$ or $(\frac{n}{2},1^{\frac{n}{2}})$ when $n$ is even. Then
			$\ccn(\chi_{\lambda};S_n)=\lceil\log_2 \lfloor\sqrt{n} \rfloor \rceil+1$.
		\end{enumerate}
	\end{theorem}
	
	\noindent The above theorem yields the covering number of the irreducible constituents of $\res_{A_n}^{S_n} \chi_{\lambda}$ when $\lambda=(\frac{n+1}{2},1^{\frac{n-1}{2}})$ is the self-conjugate hook.
	
	\begin{theorem}\label{Theorem_4}
		Assume $n\geq 5$. Let $\lambda=(\frac{n+1}{2},1^{\frac{n-1}{2}})$ when $n$ is odd and $\lambda=(\frac{n}{2}+1, 1^{\frac{n}{2}-1})$ when $n$ is even. Let $\chi$ be $\res_{A_n}^{S_n}\chi_{\lambda}$ when $n$ is even and one of $\chi_{\lambda}^{+}$ or $\chi_{\lambda}^{-}$ (where $\res_{A_n}^{S_n}\chi_{\lambda}=\chi_{\lambda}^{+} +\chi_{\lambda}^{-})$ when $n$ is odd. Then $\ccn(\chi;A_n)=\lceil\log_2 \lfloor\sqrt{n} \rfloor \rceil+1.$
	\end{theorem}

	Surprisingly, from the last two theorems, we get an irreducible character of $S_n$ whose covering number coincides with that of its restriction to $A_n$. As far as we know, this is the first instance of such a property.  Based on  \Cref{Theorem_1}, \Cref{Theorem_2}, and observations made using SageMath (\cite{sam}), we make a conjecture on the covering number of $\chi_{\lambda}$ when $\lambda$ is a two-row partition, i.e., $\lambda=(n-k,k)$, where $1\leq k\leq \lfloor \frac{n}{2}\rfloor$ (see \Cref{conjecture_two_row_partition}).
	
	\medskip

	It is natural to seek irreducible characters of finite groups with small covering numbers. In this direction, a well-known conjecture of Jan Saxl states that if $\lambda\vdash n$ is the staircase partition, that is, $\lambda=(m,m-1,\ldots,1)$ (thus $n=\frac{m(m+1)}{2}$) where $m\geq 3$, then every irreducible character appears as a constituent of $\chi_{\lambda}^2$. In other words, the Saxl conjecture states that $\ccn(\chi_{\lambda};S_n)=2$, when $\lambda$ is the staircase partition. There have been some interesting advances in this direction (see, for example, \cite{ik,ls,pgv} and references therein), and recently it has been proved by N. Harman and C. Ryba that $c(\chi_{\lambda}^3)=\irr(S_n)$ (i.e., $\ccn(\chi_{\lambda};S_n)\leq 3$) (see \cite{hr}). It is also worthwhile to mention that if $G$ is a finite simple group of Lie type and $\text{St}$ denotes the Steinberg character of $G$, then every irreducible character of $G$ appears as a constituent of $\text{St}^2$, unless $G=\text{PSU}_n(q)$ and $n\geq 3$ is co-prime to $2(q+1)$ (see \cite{hstz}). The character covering number of  $\psl_2(q)$  has been determined in \cite{pa} recently.
 
 	\medskip
 	
 	The article is organized as follows: In Sect. 2, to make the article self-contained, we discuss some basic results in the ordinary representation theory of $S_n$. In Sect. 3, we discuss some basic properties related to powers of characters of finite groups and the Kronecker product problem, which naturally appears in this scenario. In Sect. 4, we prove \Cref{Theorem_1} and \Cref{Theorem_2}. \Cref{product_of_hooks} of this section is an interesting result as well. In Sect. 5 and Sect. 6, we prove \Cref{Theorem_3} and \Cref{Theorem_4}, respectively.
 	
 	\section{Basic results on ordinary representations of $S_n$}
 	In this section, we briefly discuss some basic results in the ordinary representation theory of symmetric groups. We begin with some definitions and notations. Let  $\lambda=(\lambda_1,\lambda_2,\ldots, \lambda_l)$, where $\lambda_1\geq \cdots \geq \lambda_l$ and $\lambda_i\in \mathbb{N}$. We say $\lambda$ is a partition of $|\lambda|=n$ (written $\lambda\vdash |\lambda|=n$), where $|\lambda|=\sum_i\lambda_i=n$ is the sum of its parts. The number of parts of $\lambda$ is denoted by $l(\lambda)$. Alternatively, for a partition $\lambda$ of $n$, we write $\lambda=\langle 1^{m_1},\cdots, i^{m_i},\cdots\rangle$, where $m_i$ is the number of times $i$ occurs as a part in $\lambda$. The conjugacy classes of $S_n$ are parameterized by partitions of $n$ in the following sense: for $\pi\in S_n$, let $m(\pi)=\langle 1^{m_1},2^{m_2},\cdots \rangle$ denote the cycle-type of $\pi$. Here, $m_i$ denotes the number of $i$-cycles in the disjoint cycle decomposition of $\pi$, whence we have $\sum_i im_i=n$. Thus, $m(\pi)$ yields a partition of $n$. Suppose $\sigma,\tau \in S_n$. Then $\sigma$ and $\tau$ are conjugate in $S_n$ if and only if $m(\sigma)=m(\tau)$. This yields the desired parametrization. For $\mu\vdash n$, let $w_{\mu}$ denote the standard representative of the class of $S_n$ parameterized by $\mu$. 
 	
 	For $\lambda\vdash n$, let $S_{\lambda}=S_{\lambda_1}\times S_{\lambda_2}\times \cdots \times S_{\lambda_l}$ be a Young subgroup of $S_n$. We write $\mathbbm{1}_{G}$ for the trivial character of a finite group $G$ (or just $\mathbbm{1}$ when the group $G$ is clear from the context). Let $\sigma_{\lambda}$ be the character of the partition representation of $S_n$ indexed by $\lambda$ (see \cite[Section 2.3]{ap}). In other words, $\sigma_{\lambda}=\ind_{S_{\lambda}}^{S_n} \mathbbm{1}$. Let $\chi_{\lambda}$ denote the irreducible character of $S_n$ indexed by $\lambda$. It is well known that $\chi_{\lambda'}=\epsilon \tensor \chi_{\lambda}$, where $\epsilon(=\chi_{(1^n)})$ is the sign character of $S_n$ and $\lambda'$ is the conjugate partition of $\lambda$. Let $T_{\lambda}$ denote the Young diagram of $\lambda$. Given partitions $\lambda, \mu \vdash n$, a semi-standard Young tableaux (abbrv. SSYT) of shape $\lambda$ and type $\mu$ is a positive integer filling of the Young diagram $T_{\lambda}$ such that (a) $i$ occurs $\mu_i$ times, (b) the entries increase weakly (from left to right) along each row, and (c) the entries increase strictly down each column. The following is an example of a SSYT of shape $(4,3,2,2)$ and type $(3,3,3,1,1)$.
 	
 	\medskip
 	
 	\begin{figure}[!h]
 		\ytableausetup{smalltableaux}
 		\begin{ytableau}
 			1 & 1 & 1 & 3 \\
 			2 & 2 & 2\\
 			3 & 3\\
 			4 & 5
 		\end{ytableau}
 	\caption{A SSYT of shape $(4,3,2,2)$ and type $(3,3,3,1,1)$.}
 	\label{Fig_1}
 	\end{figure}
 	
 	\noindent  The number of SSYT of shape $\lambda$ and type $\mu$ is denoted by $K_{\lambda\mu}$, and these are called Kostka numbers. For our example above, $K_{\lambda\mu}=2$. 
 	
 	\begin{theorem}[Young's rule]\cite[Theorem 3.3.1]{ap}\label{Youngrule} For a partition $\mu\vdash n$, we have $\displaystyle \sigma_{\mu}=\sum_{\lambda\vdash n} K_{\lambda\mu}\chi_{\lambda}.$
 	\end{theorem}
 	 
 	\noindent It is well known that $K_{\lambda\mu}>0$ if and only if $\mu \trianglelefteq \lambda$, where $\trianglelefteq$ denotes the dominance order on the set of partitions of $n$ (see \cite[Lemma 3.1.12]{ap}).
 	 
 	\subsection{Symmetric functions}
 	In this subsection, we recall some basics of the theory of symmetric functions and its relation to the character theory of $S_n$. We follow the exposition in \cite[Chapter 5]{ap} (see also \cite[Chapter 1]{mac} and \cite[Chapter 7]{st}). Let $\Lambda$ denote the algebra of symmetric functions over the field $\mathbb{Q}$, and $\Lambda_n$ denotes the vector subspace of $\Lambda$ of all homogeneous symmetric functions of degree $n$.
 	The dimension of $\Lambda_n$ is the number of partitions of $n$, which we denote by $p(n)$. Let $\lambda\vdash n$. There are five fundamental bases of $\Lambda_n$: (1) the basis of monomial symmetric functions $m_{\lambda}$, (2) the basis of elementary symmetric functions $e_{\lambda}$, (3) the basis of complete symmetric functions $h_{\lambda}$, (4) the basis of power symmetric functions $p_{\lambda}$, and (5) the basis of Schur functions $s_{\lambda}$. 
 	 
 	The Frobenius characteristic function denoted by $ch_n$ relates the vector space of  class functions $R(S_n)$ of $S_n$ to $\Lambda_{n}$. It is defined by
 	$$\ch_n(f)=\frac{1}{n!}\sum_{\pi\in S_n}f(\pi)p_{m(\pi)}.$$
 	Note that the above is a $\mathbb{Q}$-linear map from $R(S_n)$ to $\Lambda_n$. The algebra $\Lambda$ is equipped with the Hall inner-product. It is  defined by imposing that the basis of Schur functions $\{s_{\lambda} \mid \lambda\vdash n\}$ form an orthonormal basis of $\Lambda_n$, that is, $\langle s_{\lambda}, s_{\mu} \rangle=\delta_{\lambda\mu}$. It is then extended bilinearly to $\Lambda$ by imposing that if $f$ and $g$ are homogeneous, then $\langle f,g\rangle\neq 0$ only if $f$ and $g$ have the same degree.  Recall that $R(S_n)$ is also equipped with an inner-product which is defined by $\displaystyle \langle f,g\rangle =\frac{1}{n!}\sum_{\pi\in S_n}f(\pi)g(\pi^{-1})$, where $f,g\in R(S_n)$.	With these forms in $R(S_n)$ and $\Lambda_n$, $\ch_n$ is an isometry, that is, for all $f,g\in R(S_n)$, $\langle \ch_n(f),\ch_n(g)\rangle = \langle f,g\rangle$. The following result gives the desired connection.
 	
 	\begin{proposition}
 	 	For partition $\lambda\vdash n $, we have $\ch_n(\chi_{\lambda})=s_{\lambda}$, $\ch_n(\sigma_{\lambda})=h_{\lambda}$, and $\ch_n(\epsilon \sigma_{\lambda})=e_{\lambda}$.
 	\end{proposition}
 
 	For partitions $\lambda, \mu \vdash n$, we recall the Murnaghan-Nakayama rule, which states that $\chi_{\lambda}(w_{\mu})=\langle p_{\mu}, s_{\lambda} \rangle$. In other words, $\displaystyle p_{\mu}=\sum_{\lambda\vdash n}\chi_{\lambda}(w_{\mu})s_{\lambda}$. Let $\mu \vdash m, \nu \vdash n$. It is a well-known result that $\ch_{m+n}(\ind_{S_m\times S_n}^{S_n}\chi_{\lambda}\otimes \chi_{\mu})=s_{\mu}s_{\nu}$, where $\chi_{\lambda}\otimes \chi_{\mu}$ is the external tensor product of $\chi_{\lambda}$ and $\chi_{\mu}$. Notice that in the RHS, the product is the usual product of symmetric functions. Thus, for a partition $\lambda \vdash m+n$, we get
 	\begin{equation}\label{LR_coeffcients}
 	 	\langle s_{\mu}s_{\nu}, s_{\lambda} \rangle =\langle \ind_{S_m\times S_n}^{S_{m+n}}\chi_{\mu}\otimes \chi_{\nu},\chi_{\lambda}\rangle=\langle \res_{S_m\times S_n}^{S_{m+n}}\chi_{\lambda},\chi_{\mu}\otimes \chi_{\nu}\rangle=c^{\lambda}_{\mu\nu}.
 	\end{equation}
 	The coefficients $c_{\mu\nu}^{\lambda}$ are called the Littlewood-Richardson coefficients (abbrv. LR-coefficients). As a consequence of \Cref{LR_coeffcients}, if $f\in \Lambda_m$ and $g\in \Lambda_n$, then
 	\begin{equation}\label{product_of_sym_interpret_repn}
 		fg=\ch_{m+n}(\ind_{S_m\times S_n}^{S_{m+n}}\chi_f\tensor\chi_g),
 	\end{equation}
 	 where $\chi_f=\ch_m^{-1}(f)\in R(S_m)$, $\chi_g=\ch_n^{-1}(g)\in R(S_n)$, and $\chi_f\tensor \chi_g \in R(S_m\times S_n)$, which is defined by $(\chi_f\tensor\chi_g)(\alpha,\beta):=\chi_f(\alpha)\chi_g(\beta)$ for all $\alpha\in S_m, \beta\in S_n$. 
 	\subsection{Littlewood-Richardson rule} We finish this section with the statement of the Littlewood-Richardson rule (abbrv. LR rule), which gives a combinatorial interpretation of the LR-coefficients. This will be needed for our computations later. Let $\lambda\vdash n$. A lattice permutation of shape $\lambda$ is a sequence of positive integers $a_1a_2\cdots a_n$, where $i$ occurs $\lambda_i$ times, and in any left factor $a_1a_2\cdots a_j$, the number of $i$'s is at least the number of $(i+1)$'s (for all $i$). For example, a lattice permutation of shape $(3,2,1)$ is 121321. A reading word of a Young tableaux $T$ is the sequence of entries of $T$ obtained by concatenating the rows of $T$ from bottom to top. For example, the reading word of the SSYT in \Cref{Fig_1} is 45332221113. For partitions $\lambda, \mu$, we say $\mu \subseteq \lambda$ if $\mu_i\leq \lambda_i$ for all $i$ (i.e., $T_{\mu}\subseteq T_{\lambda}$).  Let $T_{\lambda/\mu}$ denote the Young diagram of the (skew) shape $\lambda/\mu$. A SSYT of (skew) shape $\lambda/\mu$ and type $\nu$ is a filling of $T_{\lambda/\mu}$ with positive integer entries such that $i$ occurs $\nu_i$ times, and entries along each row from left to right (resp. along each column from top to bottom) are weakly increasing (resp. strictly increasing). \Cref{Fig_2} is an example of a SSYT of skew shape $(6,5,3,3)/(3,1,1)$ and type $(4,3,2,2,1)$. 
 	
 	 \begin{figure}[!h]
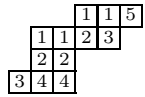

 		\centering
 		\begin{ytableau}
 			\none & \none & \none & 1 & 1 & 5\\
 			\none & 1 & 1 & 2 & 3\\
 			\none & 2 & 2\\
 			3 & 4 & 4  
 		\end{ytableau}
 		\caption{A SSYT of skew shape $(6,5,3,3)/(3,3,1)$ and type $(4,3,2,2,1)$.}
 		\label{Fig_2}
 	\end{figure}
 	
 	We are now ready to state the Littlewood-Richardson rule.

 	\begin{theorem}\cite[Littlewood-Richardson rule, Theorem A.1.3.3]{st}\label{the_LR_rule} For partitions $\mu\vdash n$, $\nu \vdash m$, and $\lambda \vdash m+n$, the LR coefficient $c_{\mu\nu}^{\lambda}$ is equal to the number of $\mathrm{SSYT}$ of shape $\lambda/\mu$ and type $\nu$ whose reverse reading word is a lattice permutation.
 	\end{theorem}
 	
 	A SSYT satisfying the condition of the above theorem is called a Littlewood-Richardson tableaux (abbrv. LR tableaux). For the tableaux in \Cref{Fig_2}, the reverse reading word (traversing top to bottom and reading right to left along a row) is 511321122443, which is not a lattice permutation. Hence, it is not a LR tableaux. We also mention two special cases of the above rule for convenience. A skew shape is called a horizontal strip of size $n$ if there are $n$ boxes, and each column has a single box. For example, if $\lambda=(4,2,1)$ and $\mu=(2,1)$, then $\lambda/\mu$ is a horizontal strip of size $4$. A skew shape is called a vertical strip of size $n$ if there are $n$ boxes, and each row has a single box. For example, if $\lambda=(3,3,3,2)$ and $\mu=(2,2,2,1)$, then $\lambda/\mu$ is a vertical strip of size $4$ (see the figures below).
 	
 	\begin{figure}[!h]
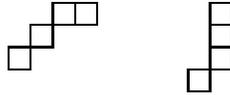

 		\centering
 		\ydiagram{2+2, 1+1,1}\;\;\;\;\;\;\;\; \ydiagram{2+1, 2+1,2+1,1+1}
 		\caption{A horizontal strip  and a vertical strip of size 4, respectively.}
 		\label{Fig_3}
 	\end{figure}
 	
 	\begin{corollary}[Pieri rule]\label{Pierirule}
 		Let $\lambda \vdash m+n$ and $\mu \vdash m$. Then
 		\[c_{\mu,(n)}^{\lambda}=
 		\begin{cases}
 			1 & \text{if } \lambda/\mu \text{ is a horizontal strip of size $n$}\\
 			0 & \text{otherwise}
 		\end{cases}.
 		\]
 	\end{corollary}
 	
 	\begin{corollary}\label{Pierirule_1}
 		Let $\lambda \vdash m+n$ and $\mu \vdash m$,
 		\[c_{\mu,(1^n)}^{\lambda}=
 		\begin{cases}
 			1 & \text{if } \lambda/\mu \text{ is a vertical strip of size $n$}\\
 			0 & \text{otherwise}
 		\end{cases}.
 		\]
 	\end{corollary}
 	
 	\noindent We also note the branching rule, which is nothing but the Pieri rule when $n=1$.
 	
 	\begin{theorem}[Branching rule]\label{branching_rule}
 		Let $\lambda \vdash n$. Then $\displaystyle \res^{S_n}_{S_{n-1}}\chi_{\lambda}=\sum_{\mu\in \lambda^{-}}\chi_{\mu}$, where $\lambda^{-}$ is the set of all partitions of $n-1$ whose Young diagram is obtained from $\lambda$ by deleting a box from $T_{\lambda}$.
 	\end{theorem}
 
 	\noindent By Frobenius reciprocity, it follows that $\displaystyle \ind^{S_{n+1}}_{S_n} \chi_{\lambda}=\sum_{\mu\in \lambda^{+}}\chi_{\mu}$, where $\lambda^{+}$ is the set of all partitions of $n+1$ whose Young diagram is obtained from $\lambda$ by adding a box to $T_{\lambda}$.
 	\section{Character covering in Symmetric groups and Kronecker product}
 	
 	In this section, we discuss the main results known on the covering number of irreducible characters of the symmetric group. The well-known Kronecker product problem is naturally related. We discuss some results in this direction, which we will use in this article.
 	
 	\subsection{The Kronecker Product} 
 	
 	Let $\times$ denote the point-wise product of two class functions in $R(S_n)$; that is, if $f,g\in R(S_n)$, then $(f\times g)(x)=f(x)g(x)$ for all $x\in G$. We write $fg$ to mean $f\times g$. For $f,g \in \Lambda_n$, the Kronecker product of $f$ and $g$, denoted by $f*g$, is defined by
 	\begin{equation}
 		f*g:=\ch_n(\ch_n^{-1}f\ch_n^{-1}g).
 	\end{equation}
 	By definition, $*$ is commutative. Further, $f*h_n=h_n*f=f$ for all $f\in \Lambda_n$. For partitions $\mu,\nu\vdash n$, from the above definition, we observe that $s_{\mu}*s_{\nu}=\ch_n(\chi_{\mu}\chi_{\nu})$. Thus, if $\displaystyle \chi_{\mu}\chi_{\nu}=\sum_{\lambda \vdash n}g_{\mu\nu\lambda}\chi_{\lambda}$, then we have $\displaystyle s_{\mu}*s_{\nu}=\sum_{\lambda\vdash n}g_{\mu\nu\lambda}s_{\lambda}$. The coefficients $g_{\mu\nu\lambda}=\langle \chi_{\mu}\chi_{\nu},\chi_{\lambda} \rangle=\langle s_{\mu}*s_{\nu},s_{\lambda} \rangle$ are called Kronecker coefficients. As the irreducible characters of $S_n$ are integer-valued, it follows that $g_{\mu\nu\lambda}$ is invariant under any permutation of $\mu,\nu,\lambda$.  The Kronecker product problem asks for a combinatorial interpretation of the Kronecker coefficients and is regarded as one of the most important open problems in the theory of symmetric groups. There are many interesting results known in this direction (for example, see \cite{bo1,bo2,bb,bk,bl,bwz,dv,gwxz,gg,li,mu1,mu4,mu2,mu3,re1,re2, rw,ro,te,va}). \cite[Section 3]{gg} discusses some elementary properties of the Kronecker product. It is convenient to extend the definition of Kronecker product to $\Lambda$ by setting $f*g=0$ if $f\in \Lambda_m, g\in \Lambda_n$, and $m\neq n$. Now we mention some results on the Kronecker product, which will be required in the sequel. 
 	
 	\medskip
 	
 	It is easy to describe the decomposition of the internal tensor product of two permutation (partition) characters $\sigma_{\lambda}$ and $\sigma_{\mu}$ of $S_n$.
 	
 	\begin{definition}
 		Let $\lambda=(\lambda_1,\ldots,\lambda_l)\vdash n$ and $\mu=(\mu_1,\ldots,\mu_m)\vdash n$. A $\lambda\times \mu$ matrix is a $l\times m$ matrix with non-negative integer entries such that  the sum of the entries of the $i$-th row equals $\lambda_i$ for all $1\leq i\leq l$, and the sum of the entries of the $j$-th column equals $\mu_j$ for all $1\leq j\leq m$.
 	\end{definition}
 	
 	\begin{theorem}\label{Kronecker_product_complete_symmetric}
 		Let $\lambda,\mu\vdash n$. Then $\displaystyle h_{\lambda}*h_{\mu}=\sum_{A=(a_{ij})} \prod_{i,j}h_{a_{ij}}$ summed over all $\lambda \times \mu$ matrices $A$.
 	\end{theorem}
 	
 	\noindent Applying the inverse Frobenius characteristic function, we obtain
 	\begin{equation}
 		\sigma_{\lambda}\sigma_{\mu}=\displaystyle \sum_{A=(a_{ij})} \prod_{i,j}\sigma_{a_{ij}},
 	\end{equation}
 	where the product in the RHS is taken in the sense of \Cref{product_of_sym_interpret_repn}. A simple representation theoretic proof of the above can be found in \cite[Theorem 10]{jk}. The following basic result of Littlewood (\cite[Theorem 3]{li}) allows one to compute the Kronecker products in principle. 
 	
 	\begin{theorem}\cite[Proposition 3.2]{gg}\label{Kroneceker_coeff_LR_coeff}
 		Let $f_1,\ldots, f_k$ be homogeneous symmetric functions of degree $a_1,\ldots, a_k$, respectively. Let $A$ be any symmetric function of degree $a_1+a_2+\cdots+a_k$. Then
 		$$(f_1f_2\cdots f_k)*A=\sum_{\alpha^{(1)}\vdash a_1}\cdots\sum_{\alpha^{(k)}\vdash a_k} \langle s_{\alpha^{(1)}}\cdots s_{\alpha^{(k)}}, A\rangle (f_1*s_{\alpha^{(1)}})\cdots (f_k*s_{\alpha^{(k)}}). $$
 		In particular,
 		$$(h_{a_1}h_{a_2}\cdots h_{a_k})*A=\sum_{\alpha^{(1)}\vdash a_1}\cdots\sum_{\alpha^{(k)}\vdash a_k} \langle s_{\alpha^{(1)}}\cdots s_{\alpha^{(k)}}, A\rangle s_{\alpha^{(1)}}\cdots s_{\alpha^{(k)}}.$$
 	\end{theorem}
 	
 	\begin{example}\label{standard_character_kronecker}
 		Let $\lambda\vdash n$ and consider the product $s_{\lambda}*s_{(n-1,1)}$. Using the previous result, $\displaystyle h_{n-1}h_1*s_{\lambda}=\sum_{\mu\vdash n-1}\sum_{\nu\vdash 1}\langle s_{\mu}s_{\nu}, s_{\lambda}\rangle s_{\mu}s_{\nu}=\sum_{\mu\vdash n-1}\langle s_{\mu}s_{1},s_{\lambda}\rangle s_{\mu}s_{1}$. Through a double application of the Pieri rule, we conclude that 
 		the above sum equals $\displaystyle \sum_{\mu \in \lambda^{-}}s_{\mu}s_1=\sum_{\nu\in \lambda^{\pm}}s_{\nu}$. Here, $\lambda^{-}$ denotes the set of partitions of $n-1$ that are obtained from $\lambda$ by removing a box from $T_{\lambda}$, and $\lambda^{\pm}$ denotes the set of all partitions of $n$ that are obtained from $\lambda$ by successive removal and addition of a box to $T_{\lambda}$. Since $h_{(n-1,1)}=s_n+s_{(n-1,1)}$, we obtain that $\displaystyle s_{\lambda}*s_{(n-1,1)}= \sum_{\mu\in \lambda^{\pm}}s_{\mu}-s_{\lambda}$. Thus, $\displaystyle \chi_{\lambda}\chi_{(n-1,1)}=\sum_{\mu\in \lambda^{\pm}}\chi_{\mu}-\chi_{\lambda}$. 
 	\end{example}

 	\subsection{Character covering number}
 	
 	Some basic facts on the products of characters can be found in \cite{ach} and \cite{mi}. We need the following lemmas, whose proofs are easy and hence omitted.
 	
 	\begin{lemma}\label{basic_lemma_1}
 		Let $G$ be a finite group, and $\chi_i$ and $\rho_i$ ($i=1,2$) be characters of $G$. Assume that $c(\chi_i)\subseteq c(\rho_i)$ for $i=1,2$. Then $c(\chi_1\chi_2)\subseteq c(\rho_1\rho_2)$.
 	\end{lemma} 
 
 	\begin{lemma}\label{basic_lemma_2}
 		Let $G$ be a finite group, and $\chi$ and $\rho$ be characters of $G$. 
 		\begin{enumerate}
 			\item If $c(\chi)=\irr(G)$, then $c(\chi\rho)=\irr(G)$.
 			
 			\item If $c(\chi)\subseteq c(\rho)$ and $c(\chi^i)=c(\rho^{i})$ for some $i\geq 1$, then $c(\chi^j)=c(\rho^j)$ for every $j\geq i$.
 		\end{enumerate}
 	\end{lemma}

 	\noindent If $\chi$ and $\rho$ are characters of group $G$ such that $c(\chi)\subseteq c(\rho)$, \Cref{basic_lemma_1} implies that $\ccn(\rho;G)\leq \ccn(\chi;G)$. For a partition $\lambda \vdash n$, we conclude that $\ccn(\sigma_{\lambda};S_n)\leq \ccn(\chi_{\lambda};S_n)$. Moreover, since $\ccn(\chi_{\lambda};S_n)=\ccn(\chi_{\lambda'};S_n)$, it follows that $\ccn(\chi_{\lambda};S_n)\geq \max\{\ccn(\sigma_{\lambda};S_n),\ccn(\sigma_{\lambda'};S_n)\}$. Thus, it is interesting to compute the character covering number of $\sigma_{\lambda}$, for it will immediately furnish a lower bound for $\ccn(\chi_{\lambda};S_n)$. Moreover, we also have the following:
 	
 	\begin{lemma}\label{basic_lemma_3}
 		Let $\lambda\vdash n$. If $c(\sigma_{\lambda}^i)=c(\chi_{\lambda}^i)$ for some $i\leq \ccn(\sigma_{\lambda};S_n)$, then $\ccn(\sigma_{\lambda};S_n)=\ccn(\chi_{\lambda};S_n)$.
 	\end{lemma}
 
 	\begin{proof}
 		The result follows immediately from \Cref{basic_lemma_2}(2).
 	\end{proof}
 
 	\Cref{Kronecker_product_complete_symmetric} provides us a recipe to compute $\ccn(\sigma_{\lambda};S_n)$. Indeed, we need to apply \Cref{Kronecker_product_complete_symmetric} repeatedly to find the least integer $k$ such that there exists a partition $\mu\vdash n$ with $\sigma_{\mu}\in c(\sigma_{\lambda}^{k-1})$ and that there exists a $\lambda\times \mu$ matrix with exactly $n$ number of entries to be 1 and all other entries 0. Notice that in this way we shall be able to find the least $k$ such that $\sigma_{(1^n)}$ (which is the regular character of $S_n$) appears as a constituent of $\sigma_{\lambda}^k$, which will imply that $\ccn(\sigma_{\lambda};S_n)=k$.
 	
 	\medskip
 	
 	Although it may not be easy to find  a closed formula for $\ccn(\sigma_{\lambda};S_n)$ for an arbitrary partition $\lambda$, in some cases it is not so difficult. Let $\mu(k)$ denote the hook partition $(n-k,1^k)$ of $n$.
 	
 	\begin{lemma}\label{covering_number_sigma_hook}
 		Let $1\leq k\leq n-2$. Then $\ccn(\sigma_{\mu(k)};S_n)=\lceil \frac{n-1}{k} \rceil$.
 	\end{lemma}
 	
 	\begin{proof}
 		Notice that when $i\geq 2$, $\sigma_{\mu(k)}^i$ has $\sigma_{\mu(n-ki)}$ as a constituent. This shows that $\ccn(\sigma_{\mu(k)})\leq \lceil \frac{n-1}{k} \rceil$. Since we are considering at each iteration a $\lambda \times \mu(k)$ matrix, starting with $\lambda=\mu(k)$,  it follows that the first part of $\lambda$ can be reduced by $k$ at the most in each iteration. This yields that $\ccn(\sigma_{\mu(k)})\geq \lceil \frac{n-1}{k} \rceil$.
 	\end{proof}
 	
 	\noindent Let $\lambda=(n-k,k)$ $(1\leq k\leq \lfloor \frac{n}{2} \rfloor)$ be a two-row partition of $n$. Clearly, $\max \{l(\mu) \mid \sigma_{\mu}\in c(\sigma_{\lambda}^i)\}\leq 2^i$. By the discussion before \Cref{covering_number_sigma_hook}, $\ccn(\sigma_{\lambda};S_n)$ is the least integer $k$ such that $\sigma_{(1^n)}\in \sigma_{\lambda}^k$, whence it follows that $2^k\geq n$. This implies that $\ccn(\sigma_{\lambda};S_n)\geq \lceil \log_2 n  \rceil$. We also have the following:
 	\begin{lemma}\label{upperbound_covering_number_sigma_tworow}
 		Let $1\leq k\leq \lfloor \frac{n}{2} \rfloor$. Then $\ccn(\sigma_{(n-k,k)};S_n)\geq \left\lceil \frac{2(n-1)}{k+1}\right\rceil$.
 	\end{lemma}
	 
	\begin{proof}
	 	Set $\lambda=(n-k,k)$. Let $(n)=\alpha_0\to \alpha_1 \to \cdots \to \alpha_r=(1^n)$ be a sequence of partitions of $n$ such that $\alpha_{i}$ is obtained by arranging the entries of a $\alpha_{i-1}\times \lambda$ matrix (which is of dimension $l(\alpha_{i-1})\times 2$) in weakly decreasing order. This implies that $\sigma_{(1^n)}\in c(\sigma_{\lambda}^{r})$.  For $1\leq i\leq r$, let $a_i=l(\alpha_i)-l(\alpha_{i-1})$ and $b_{i}=l(\alpha_{i-1}')-l(\alpha_i')$. Since $a_i$ is the increment in the number of parts of $\alpha_i$ and $b_i$ is the decrement of the largest part of $\alpha_i$, we conclude that $\sum_{i=1}^{r}a_i+b_i=2(n-1)$. Now going from the $(i-1)$-th step to the $i$-th step, if we decrease any one of the largest parts (at the $(i-1)$-th step) by $j$, where $0\leq j\leq k$, then the decrement in the largest part is at most $j$, that is, $a_i\leq j$. Also, the maximum possible increment in the number of parts is $k-j+1$, that is, $b_i\leq k-j+1$. This shows that $a_i+b_i\leq k+1$. Hence, $\sum_{i=1}^{r}a_i+b_i\leq r(k+1)\implies r\geq \frac{2(n-1)}{k+1}$. Since $\ccn(\sigma_{\lambda};S_n)$ is obtained by taking the minimum length of all the above kinds of sequences of partitions, our result follows.
	\end{proof}
 	
 	\noindent The above bound is quite crude in general (see proof of \Cref{Theorem_2}). The next proposition shows that in certain cases, the above lower bound is in fact an equality. 
 	\begin{proposition}\label{covering_number_sigma_tworow_specialcase}
 		Let $n\geq 5$ and $1\leq k\leq \sqrt{n}$. Then $\ccn(\sigma_{(n-k,k)};S_n)=\left\lceil \frac{2(n-1)}{k+1}\right\rceil$.
 	\end{proposition}
 
 	\begin{proof}
 		We construct a sequence of partitions as in the previous proof: $(n)=\alpha_0\to \alpha_1 \to \cdots \to \alpha_r=(1^n)$, where $r=\left\lceil \frac{2(n-1)}{k+1}\right\rceil$. This will prove that $\ccn(\sigma_{\lambda})\leq \left\lceil \frac{2(n-1)}{k+1}\right\rceil$, and then the proof follows from the previous lemma.  We break the proof into two cases.
 		
 		\vspace{2 mm}
 		
 		\textbf{Case 1}: Assume that $k+1\mid n-1$. Set $n-1=(k+1)b$ where $b$ is a positive integer. Then $n=bk+b+1$. Set $\alpha_i=(n-ik,k^{i})$ when $1\leq i\leq b$. Thus, $\alpha_b=(b+1,k^b)$. Since $k\leq \sqrt{n}$, $k^2\leq bk+b+1 \implies (k+1)(k-1)\leq b(k+1) \implies b\geq k-1$. Now we construct the partition  $\alpha_{b+i}$ from $\alpha_{b+i-1}$ by breaking $b-i+1$ into $b-i$ and $1$ and breaking each of the largest remaining $k-1$ parts of $\alpha_{b+i-1}$ (say $(l_1,l_2,\ldots,l_{k-1})$) into $l_j-1$ and $1$ for every $1\leq j\leq k-1$. By our algorithm, $\alpha_{r}=(1^n)$ implies that $r\geq 2b$. We show that $\alpha_{2b}=(1^n)$. To show this, we concentrate on the part of the algorithm that applies to $(\underbrace{k,k,\ldots,k}_{\text{ $b$ times}})$. For $i\geq 0$, define a  multi-set $S_i$ inductively as follows: $S_0$ consists of $b$ many $k$'s. $S_1$ consists of all those numbers that appear when the algorithm is applied to $S_0$ except for the $k-1$ many $1$'s that have been produced (these are the 1's that arise from $l_j$). In general, $S_i$ consists of all those numbers that appear when the algorithm is applied to $S_{i-1}$, except for the $k-1$ many $1$'s that have been produced. Let $i_0$ be the first $i$ such that $S_i$ consists of some $1$'s and possibly some $0$'s. Thus, $i_0$ is precisely the step where the $b$ many $k$'s have been broken down to all 1's, that is, $kb$-many 1's. We claim that $i_0=b$. In order to establish our claim, we  make two observations on the multi-set $S_i$: (i) $|S_i|=b$ and (ii) if $t_i$ is the maximum number present in $S_i$, then $S_i$ consists only of $t_i$'s and possibly $(t_i-1)$'s. Both of the properties follow from the definition of our algorithm. It follows that $t_i-1<\frac{kb-i(k-1)}{b}\leq t_i$, whence $t_i=\lceil \frac{kb-i(k-1)}{b} \rceil$. Since $b\leq k-1$, we get that  $t_{b-1}=2$ and $t_b=1$. This establishes our claim. Since by performing $b$-many iterations after $\alpha_b$, the number $b+1$ has been broken down to all $1$'s as well, we conclude that $\alpha_{2b}=(1^n)$ as desired.


 		\vspace{2 mm}
 		
 		\textbf{Case 2}: Assume $k+1\nmid n-1$. Write $n-1=(k+1)b+c$ where $0<c\leq k$. Then $n=bk+b+c+1$. Set $\alpha_i=(n-ik,k^i)$ when $1\leq i\leq b$. Thus, $\alpha_b=(b+c+1,k^b)$. Since $k\leq \sqrt{n}$, $k^2\leq bk+b+c+1 \implies (k-1)(k+1)\leq b(k+1)+c \implies k-1\leq b+\frac{c}{k+1}$. We conclude that $b\geq k-1$ since $c\leq k$. Let $\alpha_{b+1}=(b+1,k^{b-k+c}, (k-1)^{k-c},c,1^{k-c})$. Now for $i\geq 1$, $\alpha_{b+1+i}$ is constructed from $\alpha_{b+i}$ by breaking $b-i+2$ into $b-i+1$ and $1$, and breaking each of the largest remaining $k-1$ parts of $\alpha_{b+i}$ (say $l_j$) into $l_j-1$ and $1$. By our algorithm, $\alpha_r=(1^n)$ implies that $m\geq 2b+1$. Now, to find the exact value of $m$, we argue as in the previous case. For convenience, we set $\beta_{i}:=\alpha_{b+i+1}$, where $i\geq 0$. Note that $\lceil\frac{2(n-1)}{k+1}\rceil$ is $2b+1$ if $0<c\leq\lceil \frac{k}{2} \rceil$, and is $2b+2$ if $\lceil \frac{k}{2} \rceil<c\leq k$. Thus, it is enough to show that (i) $\beta_{b}=(1^n)$ if $0<c\leq\lceil \frac{k}{2} \rceil$, and (ii) $\beta_{b}\neq (1^n)$ but $\beta_{b+1}=(1^n)$ if $\lceil \frac{k}{2} \rceil<c\leq k$. To show this, we use the same method as in the previous case. We concentrate on the part of the algorithm that applies to $(k^{b-k+c},(k-1)^{k-c},c)$. For $i\geq 0$, define a  multi-set $S_i$ inductively as follows: $S_0$ consists of $b-k+c$ many $k$'s, $k-c$ many $(k-1)$'s, and $c$. $S_1$ consists of all those numbers that appear when the algorithm is applied to $S_0$ except for the $k-1$ many $1$'s that have been produced (these are the 1's that arise from $l_j$). In general, $S_i$ consists of all those numbers that appear when the algorithm is applied to $S_{i-1}$, except for the $k-1$ many $1$'s that have been produced. Let $i_0$ be the first $i$ such that $S_i$ consists of some $1$'s and possibly some $0$'s. Thus, $i_0$ is precisely the step where the $(k^{b-k+c},(k-1)^{k-c},c)$ have been broken down to all 1's, that is, $(kb-k+c)$-many 1's. We claim that $i_0=b$ if $0<c\leq\lceil \frac{k}{2} \rceil$, and $i_0=b+1$ if $\lceil \frac{k}{2} \rceil<c\leq k$. We find the average of the numbers in the set $S_i$ when $1\leq i\leq i_0$. Clearly, since $|S_i|=b+1$ for every $i$, it is easy to see that the average is equal to $\frac{(kb-k+2c)-i(k-1)}{b+1}$. Also, by the definition of our algorithm, it follows that there exists $z$ such that  $1\leq z<i_0$, and $S_z$ consists entirely of the numbers $c+1$ and $c$. Thus, by the same argument as in the previous case, we conclude that  the maximum number present in  $S_i$ is $\lceil \frac{(kb-k+2c)-i(k-1)}{b+1} \rceil$, where $z\leq i\leq i_0$.  $\lceil \frac{(kb-k+2c)-i(k-1)}{b+1} \rceil=\lceil\frac{b+2c-k}{b+1}\rceil$ when $i=b$. Clearly, $\lceil\frac{b+2c-k}{b+1}\rceil$ is 1 when $0< c\leq \lceil \frac{k}{2} \rceil$ and it takes the value 2 if $\lceil \frac{k}{2} \rceil<c\leq k$. In the latter case, $\lceil \frac{(kb-k+2c)-i(k-1)}{b+1} \rceil=1$ when $i=b+1$. This establishes our claim. Finally, starting from $\beta_0$, since the number $b+1$ has been broken down to all $1$'s at the $b$-th step, we conclude that $\beta_{b}=(1^n)$ when $0<c\leq\lceil \frac{k}{2} \rceil$ and $\beta_{b+1}=(1^n)$ when $\lceil \frac{k}{2} \rceil<c\leq k$. This completes the proof.
 		
 	\end{proof}
 	
 	\noindent For clarity, we work out some examples to demonstrate the algorithm stated in the previous proof.
 
 	\begin{example}
 		Let $\lambda=(18,3)\vdash 21$. Then $n-1=(k+1)b$, where $b=5$. We apply the algorithm demonstrated in case 1 of the proof above. We obtain the following sequence: $\alpha_0=(21)\to \alpha_1=(18,3)\to \alpha_2=(15,3^2)\to \alpha_3=(12,3^3)\to \alpha_4=(9,3^4)\to \alpha_5=(6,3^5)\to \alpha_6=(5,3,3,3,2,2,1^3)\to \alpha_7=(4,3,2,2,2,2,1^6)\to \alpha_8=(3,2,2,2,2,1^{10})\to \alpha_9=(2,2,2,1^{15})\to \alpha_{10}=(1^{20})$.
 	\end{example}

 	\begin{example}
 		Let $\lambda=(40,6)\vdash 46$. Then $n-1=(k+1)b+c$, where $b=6,c=3$. Note that $c\leq \frac{k+1}{2}$. We apply the algorithm demonstrated in case 2 of the proof above. We obtain the following sequence: $\alpha_0=(46)\to \alpha_1=(40,6)\to \alpha_2=(34,6^2)\to \alpha_3=(28,6^3)\to \alpha_4=(22,6^4)\to \alpha_5=(16,6^5)\to \alpha_6=(10,6^6)\to \alpha_7=(7,6^3,5^3,3,1^3)\to \alpha_8=(6,5^4,4^2,3,1^9)\to \alpha_9=(5,4^5,3^2,1^{15})\to \alpha_{10}=(4,3^7.1^{21})\to \alpha_{11}=(3^3,2^5,1^{27})\to \alpha_{12}=(2^5,1^{36})\to \alpha_{13}=(1^{46})$.
 	\end{example}
 	
 	\begin{example}
 		Let $\lambda=(61,8)\vdash 69$. Then $n-1=(k+1)b+c$, where $b=7,c=5$. Note that $c\geq \frac{k+1}{2}$. We once again apply the algorithm demonstrated in case 2 of the proof above. We obtain the following sequence: $\alpha_0=(69)\to \alpha_1=(61,8)\to \alpha_2=(53,8^2)\to \alpha_3=(45,8^3)\to \alpha_4=(37,8^4)\to \alpha_5=(29,8^5)\to \alpha_6=(21,8^6)\to \alpha_7=(13,8^7) \to \alpha_8=(8,8^4,7^3,5,1^3)\to \alpha_9=(7,7^4,6^3,3,1^{11})\to \alpha_{10}=(6,6^4,5^4,1^{19})\to \alpha_{11}=(5,5^5,4^3,1^{27})\to \alpha_{12}=(4,4^6,3^2,1^{35})\to \alpha_{13}=(3,3^7,2,1^{43})\to \alpha_{14}=(2,2^7,2,1^{51})\to \alpha_{15}=(2,1^{67})\to \alpha_{16}=(1^{69})$.
 	\end{example}
 	As mentioned earlier, Miller proved that $\ccn(S_n)$ is $n-1$. Using \Cref{Kronecker_product_complete_symmetric} and \Cref{standard_character_kronecker}, it is easy to see that $c(\sigma_{\mu(1)}^2)=c(\chi_{(n-1,1)}^2)$. Then \Cref{basic_lemma_3} and \Cref{covering_number_sigma_hook} at once imply that $\ccn(\chi_{(n-1,1)})=n-1$. Since $\ccn(\chi_{\lambda})=\ccn(\chi_{\lambda'})$, it follows that $\ccn(\chi_{(2,1^{n-2})})=n-1$. Thus, we get $\ccn(S_n)\geq n-1$. Miller's proof that $\ccn(S_n)\leq n-1$ (in other words, $\ccn(\chi_{\lambda};S_n)\leq n-1$ for every $\lambda\neq (n),(1^n)$) is based on the following two lemmas:

 	\begin{lemma}\cite[Lemma 11]{mi}\label{miller_lemma_1}
 		Let $\lambda$ be a non-rectangular partition of $n$. Then $c(\sigma_{\mu(2)})\subseteq c(\chi_{\lambda}^2)$.
 	\end{lemma}
 	
 	\begin{lemma}\cite[Lemma 12]{mi}\label{miller_lemma_2}
 		Let $\lambda$ be  a rectangular partition of $n$. Then $c(\sigma_{\mu(5)})\subseteq c(\chi_{\lambda}^4)$.
 	 \end{lemma}
  
  	Note that Zisser in \cite[Corollary 4.3]{zi} also proved \Cref{miller_lemma_1}. He proved it by directly computing the multiplicities of the irreducible constituents of $\sigma_{\mu(2)}$ in $\chi_{\lambda}^2$. Thus, \Cref{miller_lemma_1} is an immediate consequence of the following result. For a partition $\lambda$ of $n$, let $d_t(\lambda):=|\{i\mid \lambda_i-\lambda_{i+1}\geq t\}|$.
  	
  	\begin{lemma}\cite[Corollary 4.2]{zi}\label{formula_of_multiplicities_sigma_hook2} Let $\lambda$ be a partition of $n$. Then
  		\begin{enumerate}
  			\item $\langle \chi_{\lambda}^2,\chi_{(n)}\rangle =1$.
  			\item $\langle \chi_{\lambda}^2,\chi_{(n-1,1)}\rangle =d_1(\lambda)-1$.
  			\item $\langle \chi_{\lambda}^2,\chi_{(n-2,2)}\rangle =d_1(\lambda)(d_1(\lambda)-2)+d_2(\lambda)+d_2(\lambda').$
  			\item $\langle \chi_{\lambda}^2,\chi_{(n-2,1^2)}\rangle =(d_1(\lambda)-1)^2$.
  		\end{enumerate}
  	\end{lemma}

	\section{Proof of \Cref{Theorem_1} and \Cref{Theorem_2}} 
 	
 	The main step towards proving \Cref{Theorem_1} is an improvement of \Cref{miller_lemma_1} and \Cref{miller_lemma_2} by removing the cases $\lambda=(n-1,1), (2,1^{n-2})$ (see \Cref{multiplicity_firspart(n-3)_square_rep}). The next lemma is a well-known property of induced characters.
 	\begin{lemma}\label{induction_commutes_with_tensor}
 		Let $G$ be a finite group and $H$ be a subgroup of $G$. Suppose $\chi$ is a class function of $G$ and $\phi$ that of $H$. Then $\chi\ind_{H}^{G} \phi=\ind_{H}^{G}(\phi \res_{H}^{G}\chi).$
 	\end{lemma}
 	
 	\begin{lemma}\label{multiplicity_square_partition_rep}
 		Let $\lambda,\mu\vdash n$. Then $\langle \sigma_{\mu}, \chi_{\lambda}^2\rangle=\langle \res^{S_n}_{S_{\mu}}\chi_{\lambda}, \res^{S_n}_{S_{\mu}}\chi_{\lambda}\rangle$.
 	\end{lemma}
 	
 	\begin{proof}
 		Use \Cref{induction_commutes_with_tensor} and Frobenius reciprocity.
 	\end{proof}
 
 	\noindent The following is easy to compute using Young's rule (\Cref{Youngrule}).
 	\begin{lemma}\label{Young's_rule_application}
 		The following relations hold:
 		\begin{enumerate}
 			\item $\sigma_{(n-3,3)}=\chi_{(n-3,3)}+ \chi_{(n-2,2)}+\chi_{(n-1,1)}+\chi_{(n)}.$
 			\item $\sigma_{(n-3,2,1)}=\chi_{(n-3,2,1)}+ \chi_{(n-3,3)}+\chi_{(n-2,1^2)}+2\chi_{(n-2,2)}+2\chi_{(n-1,1)}+\chi_{(n)}.$
 			\item $\sigma_{(n-3,1^3)}=\chi_{(n-3,1^3)}+2\chi_{(n-3,2,1)}+ \chi_{(n-3,3)}+3\chi_{(n-2,1^2)}+3\chi_{(n-2,2)}+3\chi_{(n-1,1)}+\chi_{(n)}.$
 		\end{enumerate}
 	\end{lemma}
 
 	\noindent Now we are ready to state and prove the main lemma of this section.
 	\begin{lemma}\label{multiplicity_firspart(n-3)_square_rep}
 		Let $n\geq 6$ and $\lambda=(\lambda_1,\lambda_2,\ldots)\vdash n$. We have the following:
 		\begin{enumerate}
 			\item $\langle \chi_{(n-3,3)},\chi_{\lambda}^2\rangle = d_3(\lambda)+d_3(\lambda')+d_2(\lambda)(2d_1(\lambda)-3)+d_2(\lambda')(2d_1(\lambda')-3)+d_1(\lambda)(d_1(\lambda)-1)(d_1(\lambda)-3)+k+l$,
 			\item $\langle \chi_{(n-3,2,1)},\chi_{\lambda}^2\rangle = d_2(\lambda)(3d_1(\lambda)-4)+d_2(\lambda')(3d_1(\lambda')-4)+d_1(\lambda)(2d_1(\lambda)^2-8d_1(\lambda)+7)+k+l$,
 			\item $\langle \chi_{(n-3,1^3)},\chi_{\lambda}^2\rangle = d_2(\lambda)(d_1(\lambda)-1)+d_2(\lambda')(d_1(\lambda')-1)+(d_1(\lambda)-1)(d_1(\lambda)^2-3d_1(\lambda)+1)+k+l$,
 		\end{enumerate}
 	where $k=|\{i :\lambda_i=\lambda_{i+1}, \lambda_{i+1}-\lambda_{i+2}\geq 2\}|$ and $l=|\{i : \lambda_i-\lambda_{i+1}=1, \lambda_{i+1}-\lambda_{i+2}\geq 1\}|$. 
 	\end{lemma}
 
 	\begin{proof}
 	We compute $\langle \sigma_{(n-3,3)},\chi_{\lambda}^2\rangle$, $\langle\sigma_{(n-3,2,1)},\chi_{\lambda}^2\rangle$, and  $\langle \sigma_{(n-3,1^3)},\chi_{\lambda}^2\rangle$ one by one. This is equivalent to computing $\langle \res^{S_n}_{S_{n-3}\times S_3}\chi_{\lambda}$, $\res^{S_n}_{S_{n-3}\times S_3}\chi_{\lambda}\rangle$, $\langle \res^{S_n}_{S_{n-3}\times S_2}\chi_{\lambda}, \res^{S_n}_{S_{n-3}\times S_2}\chi_{\lambda}\rangle$, and $\langle \res^{S_n}_{S_{n-3}}\chi_{\lambda}, \res^{S_n}_{S_{n-3}}\chi_{\lambda}\rangle$ respectively (see \Cref{multiplicity_square_partition_rep}). We use the LR rule (\Cref{the_LR_rule}) for this computation. There are nine  possible skew shapes with three boxes, which are given in \Cref{Fig_4}. 
 		
 	\begin{figure}[!h]
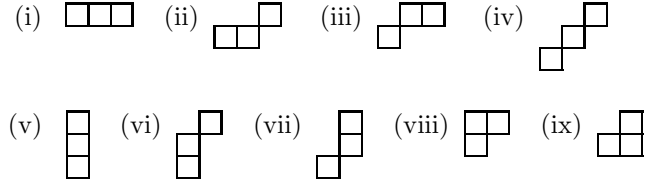

 		(i)\;\; \begin{ytableau}
 			 		 $\text{}$ &  &  \\
 				\end{ytableau}\;\;\;
 		(ii)\;\;\begin{ytableau}
 					\none &  \none &  \\
 					& 
 				\end{ytableau}\;\;\;\; 
 		(iii)\;\;\begin{ytableau}
 					\none &  &  \\
 			 		& \none & \none
 				\end{ytableau}\;\;\;\;
 		(iv)\;\;\begin{ytableau}
 					\none & \none &   \\
 					\none &  & \none\\
 					 & \none & \none
 				\end{ytableau}\;\;\;\;
 		
 		\vspace{0.5 cm}
 		
 		(v)\;\; \begin{ytableau}
 			 		\\
 			 		\\
 			 		\\
 				\end{ytableau}\;\;\;
 		(vi)\;\;\begin{ytableau}
 					\none & \\
 					\\
 					\\
 				\end{ytableau}\;\;\;
 		(vii)\;\;\begin{ytableau}
 					\none &  \\
 					\none &  \\
 					& \none
 				\end{ytableau}\;\;\;
 		(viii)\;\;\begin{ytableau}
 					 $\text{}$ &  \\
 					 & \none
 				\end{ytableau}\;\;\;
 		(ix)\;\;\begin{ytableau}
 					\none & \\
 				     	  &
 				\end{ytableau}
 		\caption{Possible skew shapes with three boxes.}
 		\label{Fig_4}
 	\end{figure}
 	Given $\lambda\vdash n$, we compute the number of partitions $\mu\vdash n-3$ such that $\lambda/\mu$ is given by each of these shapes. Clearly, the number of $\mu's$ for shape (i) is $d_3(\lambda)$ and that of shape (v) is $d_3(\lambda')$. The number of $\mu's$ such that $\lambda/\mu$ has shape given by one of (ii) or (iii) is $d_2(\lambda)(d_1(\lambda)-1)$. Since shapes in (vi) and (vii) are conjugates of (ii) and (iii), respectively, the  number of $\mu's$ such that $\lambda/\mu$ is given by one of (vi) or (vii) is $d_2(\lambda')(d_1(\lambda')-1)$. The number of $\mu's$ for shape (iv) is clearly $\binom{d_1(\lambda)}{3}$. Finally, it is not difficult to see that the number of $\mu's$ for the shape in (viii) and (ix) is given by $l$ and $k$, respectively.
 	
 	\medskip
 	
 	We start with the evaluation of $\langle \res^{S_n}_{S_{n-3}\times S_3}\chi_{\lambda}$, $\res^{S_n}_{S_{n-3}\times S_3}\chi_{\lambda}\rangle$. We have,
 	$$\displaystyle \res_{S_{n-3}\times S_3}^{S_n}\chi_{\lambda}=\sum_{\mu\vdash n-3}c_{\mu,(3)}^{\lambda}(\chi_{\mu}\tensor\chi_{(3)})+\sum_{\mu\vdash n-3}c_{\mu,(2,1)}^{\lambda}(\chi_{\mu}\tensor\chi_{(2,1)})+\sum_{\mu\vdash n-3}c_{\mu,(1^3)}^{\lambda}(\chi_{\mu}\tensor\chi_{(1^3)}).$$
	Using the Pieri rule (see \Cref{Pierirule}), $c_{\mu,(3)}^{\lambda}=1$ if $\lambda/\mu$ is a horizontal strip, that is, $\lambda/\mu$ has shape given by one of (i)-(iv), otherwise $c_{\mu,(3)}^{\lambda}=0$. Thus, $\sum\limits_{\mu \vdash n-3} [c_{\mu,(3)}^{\lambda}]^2=d_3(\lambda)+d_2(\lambda)(d_1(\lambda)-1)+\binom{d_1(\lambda)}{3}$. Similarly, by \Cref{Pierirule_1}, $c_{\mu,(1^3)}^{\lambda}=1$ if $\lambda/\mu$ is a vertical strip, that is, $\lambda/\mu$ has shape given by one of (iv)-(vii), otherwise $c_{\mu,(1^3)}^{\lambda}=0$. Thus, $\sum\limits_{\mu \vdash n-3}[c_{\mu,(1^3)}^{\lambda}]^2=d_3(\lambda')+d_2(\lambda')(d_1(\lambda')-1)+\binom{d_1(\lambda)}{3}$. Now we compute $c_{\mu,(2,1)}^{\lambda}$. Using the LR rule, if $\mu$ is such that $\lambda/\mu$ is one of the shapes in (i)-(ix), then $c_{\mu,(2,1)}^{\lambda}$ is the number of SSYTs of that shape and type $(2,1)$ such that the reverse reading word is a lattice permutation. We have the following:
	\[c_{\mu,(2,1)}^{\lambda}=
	\begin{cases}
		1 & \text{if $\lambda/\mu$ has shape as in (ii), (iii), and (vi)-(ix)},\\
		2 & \text{if $\lambda/\mu$ has the shape in (iv)},\\
		0 & \text{otherwise}.
	\end{cases}
	\]
	\noindent Thus, $\sum\limits_{\mu \vdash n-3}[c_{\mu,(2,1)}^{\lambda}]^2=d_2(\lambda)(d_1(\lambda)-1)+d_2(\lambda')(d_1(\lambda')-1)+4\binom{d_1(\lambda)}{3}$+k+l. We conclude that
	\begin{eqnarray*}
		&& \langle \sigma_{(n-3,3)},\chi_{\lambda}^2\rangle=\sum_{\mu \vdash n-3} [c_{\mu,(3)}^{\lambda}]^2+[c_{\mu,(2,1)}^{\lambda}]^2+[c_{\mu,(1^3)}^{\lambda}]^2\\ &=&d_3(\lambda)+d_3(\lambda')+2d_2(\lambda')(d_1(\lambda')-1)+2d_2(\lambda)(d_1(\lambda)-1)+6\binom{d_1(\lambda)}{3}+k+l.
	\end{eqnarray*} 
	By \Cref{Young's_rule_application}, $\chi_{(n-3,3)}=\sigma_{(n-3,3)}-\chi_{(n-2,2)}-\chi_{(n-1,1)}-\chi_{(n)}$, and hence $(1)$ follows from \Cref{formula_of_multiplicities_sigma_hook2}.
	 	
 	\medskip
 	
 	Now we compute $\langle \res_{S_{n-3}\times S_2}^{S_n} \chi_{\lambda}, \res_{S_{n-3}\times S_2}^{S_n} \chi_{\lambda}\rangle$. Let $\displaystyle \res_{S_{n-3}\times S_2}^{S_n} \chi_{\lambda}=\sum_{\mu\vdash n-3}a_{\mu}(\chi_{\mu}\tensor\chi_{(2)})+\sum_{\mu\vdash n-3}b_{\mu}(\chi_{\mu}\tensor\chi_{(1^2)}).$ At first, we compute $a_{\mu}$. The skew shapes with two boxes are given by 
 	\begin{figure}[!h]
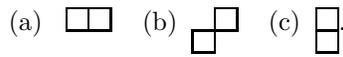

 		(a)\;\; \begin{ytableau}
 					$\text{}$ &   \\
 				\end{ytableau}\;\;\;
 		(b)\;\;\begin{ytableau}
 					\none &   \\
 					& \none
 				\end{ytableau}\;\;\;
 		(c)\;\;\begin{ytableau}
 					\\
 					\\
 				\end{ytableau}.
 		\caption{Possible skew shapes with two boxes.}
 		\label{Fig_5}
 	\end{figure}
 
 	Applying the branching rule and Pieri rule, we note that given a $\mu$ such that $\lambda/\mu$ is one of the skew shapes in (i)-(ix) of \Cref{Fig_4}, the value of $a_{\mu}$ is equal to the number of partitions $\nu\vdash n-1$ such that $\mu \subset \nu \subset\lambda$ with $\lambda/\nu$ being a single box and $\nu/\mu$ being a horizontal strip with two boxes, that is, the skew shapes (a) and (b) in \Cref{Fig_5}. For example, if $\lambda/\mu$ is the skew shape in \Cref{Fig_4}(ii), then the number of distinct  $\nu$'s with the above property is two, whence $a_{\mu}=2$. The following colored diagram explains this.
 	
 	
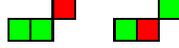
\begin{figure}[!h]
 		\begin{ytableau}
 			\none &  \none & *(red) \\
 			 *(green) & *(green)
 		\end{ytableau}\;\;\;\; 
 		\begin{ytableau}
 			\none &  \none & *(green) \\
 			 *(green) & *(red) 
 		\end{ytableau}\;\;\;\; 
 	\end{figure}
 	In both of the diagrams above, $\nu$ is such that $\lambda/\nu$ is the red colored box and $\nu/\mu$ is the green colored horizontal strip. More generally, we have the following:
 	\[a_{\mu}=
 	\begin{cases}
 		1 & \text{if $\lambda/\mu$ has shape as in (i), (vi)-(ix)},\\
 		2 & \text{if $\lambda/\mu$ has the shape as in (ii) and (iii)},\\
 		3 & \text{if $\lambda/\mu$ has the shape as in (iv)},\\
 		0 & \text{otherwise}.
 	\end{cases}
 	\]
 	Thus, $\sum\limits_{\mu \vdash n-3} a_{\mu}^2=d_3(\lambda)+4d_2(\lambda)(d_1(\lambda)-1)+d_2(\lambda')(d_1(\lambda')-1)+k+l+9\binom{d_1(\lambda)}{3}$.
 	Similarly, we can compute the value of $b_{\mu}$. In this case, we have to take into account the vertical strips with two boxes, that is, skew shapes (b) and (c) of \Cref{Fig_5}. We have the following:
 	\[b_{\mu}=
 	\begin{cases}
 		1 & \text{if $\lambda/\mu$ has shape as in (ii), (iii), (v), (viii), and (ix)},\\
 		2 & \text{if $\lambda/\mu$ has the shape as in (vi) and (vii)},\\
 		3 & \text{if $\lambda/\mu$ has the shape as in (iv)},\\
 		0 & \text{otherwise}.
 	\end{cases}
 	\]
 	Thus, $\sum\limits_{\mu \vdash n-3}b_{\mu}^2=d_3(\lambda')+ d_2(\lambda)(d_1(\lambda)-1)+4d_2(\lambda')(d_1(\lambda')-1)+k+l+9\binom{d_1(\lambda)}{3}$. Therefore, we conclude that
 	\begin{eqnarray*}
 		&& \langle \sigma_{(n-3,2,1)},\chi_{\lambda}^2\rangle=\sum_{\mu \vdash n-3} a_{\mu}^2+b_{\mu}^2\\ &=& d_3(\lambda)+d_3(\lambda')+5d_2(\lambda')(d_1(\lambda')-1)+5d_2(\lambda)(d_1(\lambda)-1)+18\binom{d_1(\lambda)}{3}+2k+2l.
 	\end{eqnarray*} 
 	By \Cref{Young's_rule_application}, $\chi_{(n-3,2,1)}=\sigma_{(n-3,2,1)}-\chi_{(n-3,3)}-\chi_{(n-2,1^2)}-2\chi_{(n-2,2)}-2\chi_{(n-1,1)}-\chi_{(n)}$. Thus, by \Cref{formula_of_multiplicities_sigma_hook2}, (2) follows.
 	
 	\medskip
 	
 	Finally, we compute $\langle \res_{S_{n-3}}^{S_n}\chi_{\lambda}, \res_{S_{n-3}}^{S_n}\chi_{\lambda}\rangle$. We have,
 	$$\res_{S_{n-3}}^{S_n}\chi_{\lambda}=\sum_{\mu\vdash n-3}a_{\mu}\chi_{\mu}.$$
 	By applying the branching rule, we note that if $\mu\vdash n-3$ is such that $\lambda/\mu$ is one of the skew shapes (i)-(ix) of \Cref{Fig_4}, then $a_{\mu}$ is the cardinality of the set of all tuples $(\nu,\eta)$, where $\eta\vdash n-1$, $\nu\vdash n-2$ are such that $\mu\subset\nu\subset\eta\subset\lambda$ with $\lambda/\eta$, $\eta/\nu$, $\nu/\mu$ each being a single box. For example, if $\mu$ is such that $\lambda/\mu$ is the skew shape in (ii), then the number of different tuples $(\mu,\eta)$ with the desired property is three. The following colored diagram explains this.
 	
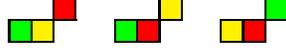
\begin{figure}[!h]
 		\begin{ytableau}
 			\none &  \none & *(red) \\
 			*(green) & *(yellow)
 		\end{ytableau}\;\;\;\; 
 		\begin{ytableau}
 			\none &  \none & *(yellow) \\
 			*(green) & *(red) 
 		\end{ytableau}\;\;\;\; 
 		\begin{ytableau}
 			\none &  \none & *(green) \\
 			*(yellow) & *(red) 
 		\end{ytableau}
 	\end{figure}
 
 	\noindent In the above figure, $\lambda/\eta$ is denoted by the red box, $\eta/\nu$ is denoted by the yellow box, and $\nu/\mu$ is denoted by the green box. This way of counting gives the following:
 	\[a_{\mu}=
 	\begin{cases}
 		1 & \text{if $\lambda/\mu$ has shape as in (i) and (v)},\\
 		2 & \text{if $\lambda/\mu$ has the shape as in (viii) and (ix)},\\
 		3 & \text{if $\lambda/\mu$ has the shape as in (ii), (iii), (vi), and (vii)},\\
 		6 & \text{if $\lambda/\mu$ has the shape as in (iv)}.
 	\end{cases}
 	\]
 	Thus, we conclude that
 	\begin{eqnarray*}
 		&& \langle \sigma_{(n-3,1^3)},\chi_{\lambda}^2\rangle=\sum_{\mu \vdash n-3} a_{\mu}^2\\ &=& d_3(\lambda)+d_3(\lambda')+9d_2(\lambda')(d_1(\lambda')-1)+9d_2(\lambda)(d_1(\lambda)-1)+36\binom{d_1(\lambda)}{3}+4k+4l.
 	\end{eqnarray*} 
 	By \Cref{Young's_rule_application}, $\displaystyle \chi_{(n-3,1^3)}=\sigma_{(n-3,1^3)}-2\chi_{(n-3,2,1)}-\chi_{(n-3,3)}-3\chi_{(n-2,1^2)}-3\chi_{(n-2,2)}-3\chi_{(n-1,1)}-\chi_{(n)}$, and hence $(3)$ also follows from \Cref{formula_of_multiplicities_sigma_hook2}. The proof is now complete.
	\end{proof}
	
	\begin{remark}
		The previous proof can be written completely in the language of symmetric functions. We recall that for partitions $\mu,\lambda$ such that $\mu\subseteq \lambda$, $s_{\lambda/\mu}$ is the skew-Schur function corresponding to the skew shape $\lambda/\mu$. Further, if $\lambda\vdash m+n, \mu \vdash m$, and $\nu\vdash n$ then
		$$c^{\lambda}_{\mu\nu}=\langle s_{\mu}s_{\nu}, s_{\lambda} \rangle=\langle s_{\mu},s_{\lambda/\nu} \rangle = \langle s_{\nu}, s_{\lambda/\mu}\rangle.$$
		
		We note that $\langle \sigma_{\mu},\chi_{\lambda}^2\rangle$ is equal to the coefficient of $s_{\lambda}$ in the Schur expansion of $h_{\mu}*s_{\lambda}$. A consequence of  \Cref{Kroneceker_coeff_LR_coeff} is the following result (see \cite[Theorem 3.1]{gg}). 
		\begin{equation}\label{complete_schur_kronecker_expansion}
					h_{\mu}*s_{\lambda}=\sum\prod_{i\geq 1}s_{\lambda^i/\lambda^{i-1}},
		\end{equation}
		summed over all sequences $(\lambda^0,\lambda^1,\cdots, \lambda^l)$ of partitions such that $\phi=\lambda^0\subseteq \lambda^1\subseteq \cdots \subseteq \lambda^l=\lambda$ and $|\lambda^i/\lambda^{i-1}|=\mu_i$ for all $i\geq 1$. Using this decomposition and the  LR rule, finding the coefficient of $s_{\lambda}$ boils down to the same computations that we have done above.
	\end{remark}
	\begin{lemma}\label{miller_lemma_1_improved}
		Let $n\geq 6$ and $\lambda$ be a non-rectangular partition of $n$ such that $\lambda\notin \{(4,2),(2^2,1^2),(n-1,1), (2,1^{n-2})\}$. Then $c(\sigma_{\mu(3)})\subseteq c(\chi_{\lambda}^2)$.
	\end{lemma}

	\begin{proof}
		We need to show that $\{ \chi_{(n-3,3)}, \chi_{(n-3,2,1)}, \chi_{(n-3,1^3)} \} \subseteq c(\chi_{\lambda}^2)$. We use the multiplicity results of \Cref{multiplicity_firspart(n-3)_square_rep}. We also use the fact that $d_1(\lambda)=d_1(\lambda')$. Assume first that $d_1(\lambda)\geq 3$. Note that the quadratic functions $2x^2-8x+7$ and $x^2-3x+1$ are both strictly positive in the interval $[3,\infty)$. Therefore, it easily follows that $\langle \chi_{(n-3,2,1)}, \chi_{\lambda}^2\rangle >0$ and $\langle \chi_{(n-3,1^3)}, \chi_{\lambda}^2\rangle >0$. If $d_2(\lambda)$ or $d_2(\lambda')$ is positive, then $\langle \chi_{(n-3,3)}, \chi_{\lambda}^2\rangle>0$. Thus, we may assume that both $d_2(\lambda)$ and $d_2(\lambda')$ are zero. If $d_1(\lambda)\geq 4$, we once again conclude that $\langle \chi_{(n-3,3)},\chi_{\lambda}^2\rangle>0$. Now assume that $d_1(\lambda)=3$. It is clear that the only choice is  $\lambda=(3,2,1)$ in this case. If $\lambda=(3,2,1)$, then $l=1$. We conclude that $\langle \chi_{(n-3,3)}, \chi_{\lambda}^2\rangle>0$. Overall, we conclude that if $d_1(\lambda)\geq 3$, then $c(\sigma_{\mu(3)})\subseteq c(\chi_{\lambda}^2)$. Now we may assume that $d_1(\lambda)=2$. We have the following:
		\begin{enumerate}
			\item $\langle \chi_{(n-3,3)}, \chi_{\lambda}^2\rangle =d_3(\lambda)+d_3(\lambda')+d_2(\lambda)+d_2(\lambda')+l+k-2$.
			\item $\langle \chi_{(n-3,2,1)}, \chi_{\lambda}^2\rangle = 2d_2(\lambda)+2d_2(\lambda')+l+k-2$.
			\item $\langle \chi_{(n-3,1^3)}, \chi_{\lambda}^2\rangle = d_2(\lambda)+d_2(\lambda')+l+k-1$.
		\end{enumerate}
		Since $d_1(\lambda)=2$, we write $\lambda=(a^x,b^y)$, where $x,y\geq 1$ and $a>b$. If $b\geq 2$ and $y\geq 2$, then both $d_2(\lambda)$ and $d_2(\lambda')$ are positive. Moreover, $k\geq 1$. We conclude that each of the inner-products above is positive, and hence $c(\sigma_{\mu(3)})\subseteq c(\chi_{\lambda}^2)$ as required. Notice that  $\lambda'=((x+y)^b, x^{a-b})$. The above argument applies to $\lambda'$ when $x\geq 2$ and $a-b\geq 2$, whence we can conclude that $c(\sigma_{\mu(3)})\subseteq c(\chi_{\lambda'}^2)$. Since $\chi_{\lambda}^2=\chi_{\lambda'}^2$, we conclude that $c(\sigma_{\mu(3)})\subseteq c(\chi_{\lambda}^2)$. Therefore, we can assume that $\lambda$ satisfies (i) $b=1$ or $y=1$ and (ii) $a-b=1$ or $x=1$. It is enough to show that all the inner-products from (1)-(3) are positive in these remaining cases as well. We do this case-by-case.
		
		\medskip
		
		\textbf{Case I}: Assume that $b=1$ and $a=b+1=2$, whence $\lambda=(2^x,1^y)$. Note that since $n\geq 6$ and $\lambda\neq (2,1^{n-2})$, we have $x\geq 2$ and $y\geq 2$. Moreover, as $\lambda\neq (2^2,1^2)$, at least one of $x$ or $y$ is greater than two. We conclude that $d_3(\lambda')\geq 1$ and $d_2(\lambda')=2$, whence all the inner-products from (1)-(3) are positive.
		
		\textbf{Case II}: Assume that $b=1$ and $x=1$, whence $\lambda=(a,1^y)$. Since $\lambda\neq (n-1,1), (2,1^{n-2})$, we can assume $y\geq 2$ and $a\geq 3$. Clearly, $d_2(\lambda)=2$ and $d_2(\lambda')=1$. Once again, all the inner-products from (1)-(3) are positive.
		
		\textbf{Case III}: Assume that $y=1$ and $a=b+1$, whence $\lambda=((b+1)^x,b)$. Notice that $l=1$ in this case. Suppose $x>2$. Then $d_{3}(\lambda')\geq 1$, which also implies $d_2(\lambda')\geq 1$, and we conclude that all inner-products from (1)-(3) are positive. Thus, we may assume that $x\leq 2$. If $x=2$, then $\lambda=(b+1,b+1,b)$. Since $n\geq 6$, $b\geq 2$. Thus, $d_2(\lambda)=1$ and $d_2(\lambda')=1$, whence all inner-products from (1)-(3) are positive once again. If $x=1$, $\lambda=(b+1,b)$. Since $n\geq 6$, $b\geq 3$. In this case $d_3(\lambda)=1$, which implies that $d_2(\lambda)\geq 1$. We once again conclude that the inner-products from (1)-(3) are positive.

		\textbf{Case IV}: Assume that $y=1$ and $x=1$, whence $\lambda=(a,b)$. Clearly $b\geq 2$. If $b\geq 3$, then $d_3(\lambda)\geq 1$ which implies $d_2(\lambda)\geq 1$. Further if $a\geq b+2$, note that $d_2(\lambda)=2$. Once again all the inner-products from (1)-(3) are positive. Finally, if $a=b+1$, then $l=1$, whence the inner-products from (1)-(3) are positive once again. We can assume that $b=2$. Since $n\geq 6$ and $\lambda\neq (4,2)$, we assume $a\geq 5$. In this case, $d_3(\lambda)=1$ and $d_2(\lambda)=2$, whence all the inner-products from (1)-(3) are positive.
	\end{proof}
	
	\begin{lemma}\label{miller_lemma_2_improved}
		Let $n\geq 12$, $\lambda$ be a rectangular partition of $n$, and $\lambda\neq (n),(1^n)$. Then $c(\sigma_{\mu(6)})\subseteq c(\chi_{\lambda}^4)$.
	\end{lemma}

	\begin{proof}
		Let $\lambda=(r^s)\vdash n$, where $r,s\geq 2$. Note that $d_i(\lambda)=d_i(\lambda')=1$ for $i=1,2$. Also, $d_{3}(\lambda)=1$ if $r\geq 3$ and $d_3(\lambda')=1$ if $s\geq 3$. Finally, $k=1$ and $l=0$. Thus, \Cref{multiplicity_firspart(n-3)_square_rep} yields (i) $\langle \chi_{(n-3,1^3)}, \chi_{\lambda}^2 \rangle =1$, (ii) $\langle \chi_{\lambda}^2,\chi_{(n-3,3)}\rangle=1$ unless $r=2$ or $s=2$, and (iii) $\langle \chi_{\lambda}^2,\chi_{(n-3,2,1)}\rangle=0$. Using \Cref{formula_of_multiplicities_sigma_hook2}, $\langle \chi_{\lambda}^2,\chi_{(n-2,2)}\rangle=1$, $\langle \chi_{\lambda}^2, \chi_{(n-1,1)}\rangle=\langle \chi_{\lambda}^2, \chi_{(n-1,1^2)}\rangle=0$. Using Young's rule, we observe that 
		\[\langle \chi_{\lambda}^2,\chi_{(n-4,4)}\rangle=\langle \chi_{\lambda}^2,\sigma_{(n-4,4)}\rangle-\langle \chi_{\lambda}^2,\sigma_{(n-3,3)}\rangle=\begin{cases}
			\langle \chi_{\lambda}^2,\sigma_{(n-4,4)}\rangle-2 & \text{if $r=2$, or $s=2$}\\
			\langle \chi_{\lambda}^2,\sigma_{(n-4,4)}\rangle-3 & \text{otherwise}
		\end{cases}.\]
		By \Cref{multiplicity_square_partition_rep}, $\langle \chi_{\lambda}^2,\sigma_{(n-4,4)}\rangle=\langle \res_{S_{n-4}\times S_4}^{S_n}\chi_{\lambda},\res_{S_{n-4}\times S_4}^{S_n}\chi_{\lambda}\rangle$. To compute this inner-product, we need to  compute $c^{\lambda}_{\mu\nu}$ where $\nu \vdash 4$ (see \Cref{LR_coeffcients}). Since $\lambda$ is rectangular, for any $\mu\subseteq \lambda$, the skew diagram $T_{\lambda/\mu}$ must be right justified. Thus, using \Cref{the_LR_rule}, we obtain the following:
		\begin{itemize}
			\item there exists $!\mu \vdash n-4$ such that $c^{\lambda}_{\mu,(4)}=1$ provided $r\geq 4$, otherwise $c_{\mu,(4)}^{\lambda}=0$.
			\item there exists $!\mu \vdash n-4$ such that $c^{\lambda}_{\mu,(1^4)}=1$ provided $s\geq 4$, otherwise $c_{\mu,(1^4)}^{\lambda}=0$.
			\item there exists $!\mu \vdash n-4$ such that $c^{\lambda}_{\mu,(3,1)}=1$ provided $r\geq 3$, otherwise $c_{\mu,(3,1)}^{\lambda}=0$.
			\item there exists $!\mu \vdash n-4$ such that $c^{\lambda}_{\mu,(2,2)}=1$ provided $r\geq 3$ or $s\geq 3$, otherwise $c_{\mu,(2,2)}^{\lambda}=0$.
			\item there exists $!\mu \vdash n-4$ such that $c^{\lambda}_{\mu,(2,1,1)}=1$ provided $s\geq 3$, otherwise $c_{\mu,(2,1,1)}^{\lambda}=0$.
		\end{itemize}
	Since $n\geq 12$, from the above informations, it is easy to conclude that $\langle \chi_{\lambda}^2,\chi_{(n-4,4)}\rangle>0$. This yields that $c(\chi_{\lambda}^4)\supseteq c(\chi_{\mu}\chi_{\nu})$, where $\mu$ and $\nu$ can be chosen from the set $\{(n-2,2), (n-3,1^3), (n-4,4)\}$. Due to \Cref{miller_lemma_2}, to get the final result, it is enough to show that $\chi_{\nu}\in c(\chi_{\lambda}^4)$ for each $\nu\vdash n$ with the first part equal to $n-6$. By \cite[Section 2, Eqn. 24,53,54]{mu1}, we conclude that all such $\chi_{\nu}$'s occur with the possible exception of $\chi_{(n-6,3^2)}$. On the other hand, using the fact that $\chi_{(n-4,4)}^2=\sigma_{(n-4,4)}^2-2\sigma_{(n-4,4)}\sigma_{(n-3,3)}+\sigma_{(n-3,3)}^2$ (see \Cref{Youngrule}), an easy check using \Cref{Kronecker_product_complete_symmetric} yields that $\chi_{(n-6,3^2)}\in c(\chi_{(n-4,4)}^2)$, whence our proof is complete.
	\end{proof}
	For $\lambda, \mu\vdash n$, let $|\lambda \setminus \mu|$ be the number of boxes of $T_{\lambda}$ that lie outside of $T_{\mu}$. For example, if $\lambda=(4,3,2,2)$ and $\mu=(3,3,3,1,1)$, then $\mid \lambda \setminus \mu \mid =2$. The next lemma is an easy application of \Cref{complete_schur_kronecker_expansion} and the Pieri rule.
	
	\begin{lemma}\cite[Lemma 10]{mi}\label{hook_constituents}
		Let $\lambda\vdash n$, $1\leq k\leq n-2$, and $r\in \mathbb{N}$. Then $c(\sigma_{\mu(k)}^r\chi_{\lambda})=\{\chi_{\mu}\mid  |\lambda \setminus \mu|\leq kr\}.$
	\end{lemma}
	
	\begin{proof}
		We prove it when $r=1$. By \Cref{complete_schur_kronecker_expansion}, 
		$$h_{\mu(k)}*s_{\lambda}=\sum s_{\nu}s_1^k,$$
		where the sum runs over all sequences $\nu_0 \subseteq \nu_1 \subseteq \cdots \subseteq \nu_k=\lambda$ such that $\nu_0=\nu\subseteq \lambda$ is a partition of $n-k$, and $|\nu^i/\nu^{i-1}|=1$ for every $1\leq i\leq k$. Using the Pieri rule, we conclude that if $\langle s_{\nu}s_1^k,s_{\mu}\rangle >0$, then $|\lambda\setminus \mu|\leq k$. Conversely, if $\mu \vdash n$ is such that $|\lambda\setminus \mu|\leq k$, then $\mu\cap \lambda$ is a partition of $t$ where $t\geq n-k$. Thus, we can choose $\nu\vdash n-k$ such that $\nu \subseteq \lambda\cap \mu \subseteq \lambda$, and hence $s_{\nu}s_1^k$ appears in the above summand. Once again the Pieri rule implies that $\langle s_{\nu}s_1^k,s_{\mu}\rangle >0$. This yields the result for $r=1$. For $r\geq 2$, the result follows by induction.
	\end{proof}
	\begin{proof}[\textbf{Proof of \Cref{Theorem_1}}]
		If $n\leq 11$, the theorem can be verified directly using SageMath (\cite{sam}). So we may assume $n\geq 12$. Let $\lambda\vdash n$ be non-rectangular. Set $n=3k+r$, where $0\leq r\leq 2$ and $k\geq 3$. Then $c(\sigma_{\mu(3)}^k)\subseteq c(\chi_{\lambda}^{2k})$. If $r\in \{0,1\}$, then by \Cref{covering_number_sigma_hook}, $c(\sigma_{\mu(3)}^k)=\irr(S_n)$ and hence $c(\chi_{\lambda}^{2k})=\irr(S_n)$, whence the result follows. If $r=2$, notice that $c(\sigma_{\mu(3)}^k\chi_{\lambda})\subseteq c(\chi_{\lambda}^{2k+1})$ and $c(\sigma_{\mu(3)}^k\chi_{\lambda})=\irr(S_n)$ by \Cref{hook_constituents}. Hence $c(\chi_{\lambda}^{2k+1})=\irr(S_n)$ and once again the result follows. Now we consider the rectangular partitions. Let $n=6k+r$ where $0\leq r\leq 5$ and $k\geq 1$. By \Cref{miller_lemma_2_improved}, we obtain 
		\begin{equation}\label{main_proof_eqn}
			c(\sigma_{\mu(6)}^k)\subseteq c(\chi_{\lambda}^{4k}).
		\end{equation}
		If $r\in \{0,1\}$, then by \Cref{covering_number_sigma_hook}, $c(\sigma_{\mu(6)}^k)=\irr(S_n)$ which implies $c(\chi_{\lambda}^{4k})=\irr(S_n)$, whence the result follows. If $r=2$, then $\lceil \frac{2(n-1)}{3} \rceil=4k+1$. From \Cref{main_proof_eqn}, we get $c(\sigma_{\mu(6)}^k\chi_{\lambda})\subseteq c(\chi_{\lambda}^{4k+1})$.
		By \Cref{hook_constituents}, it is easy to check that $c(\sigma_{\mu(6)}^k\chi_{\lambda})=\irr(S_n)$, from which the result once again follows. Now we assume that $r\in \{3,4,5\}$. We have seen in the proof of \Cref{miller_lemma_2_improved} that $\chi_{(n-3,1^3)}\in c(\chi_{\lambda}^2)$. Thus, from \Cref{main_proof_eqn}, we get $c(\sigma_{\mu(6)}^k\chi_{(n-3,1^3)})\subseteq c(\chi_{\lambda}^{4k+2})$. Once again using \Cref{hook_constituents} and the fact that $n\geq 12$, it is easy to deduce that $c(\sigma_{\mu(6)}^k\chi_{(n-3,1^3)})=\irr(S_n)$ when $r\in \{3,4\}$. Moreover, when $r=5$, we get that $c(\sigma_{\mu(6)}^k\chi_{(n-3,1^3)})=\irr(S_n)\setminus\{\epsilon\}$. Since $\chi_{(4,1^{n-4})}\in c(\sigma_{\mu(6)}^k)$ (once again by \Cref{hook_constituents}), we conclude that $\epsilon \in c(\sigma_{\mu(6)}^k\chi_{(n-3,1^3)})$. Overall, if $r\in \{3,4,5\}$, we get that $c(\chi_{\lambda}^{4k+2})\supseteq c(\sigma_{\mu(6)}^k\chi_{(n-3,1^3)})=\irr(S_n)$. Since $\lceil \frac{2(n-1)}{3} \rceil=4k+2$ when $r=3,4$ and $\lceil \frac{2(n-1)}{3} \rceil=4k+3$ when $r=5$, our result follows.
		
		\medskip
		
		The final assertion follows since $\ccn(\chi_{(n-2,2)};S_n)\geq \ccn(\sigma_{(n-2,2)};S_n)=\left\lceil \frac{2(n-1)}{3} \right\rceil$ by \Cref{upperbound_covering_number_sigma_tworow}. The fact that $\ccn(\chi_{(2^2,1^{n-4})};S_n)=\left\lceil \frac{2(n-1)}{3} \right\rceil$ follows since $\ccn(\chi_{\lambda};S_n)=\ccn(\chi_{\lambda'};S_n)$.
	\end{proof}
	
	\begin{proof}[\textbf{Proof of \Cref{Theorem_2}}]
		Let $k=\frac{n+1}{2}$ and $\lambda=(k,k-1)$. By \cite[Corollary 4.1]{gwxz},  we get that $\displaystyle \chi_{\lambda}^2=\sum_{l(\mu)\leq 4}\chi_{\mu}$. It follows from \Cref{Kronecker_product_complete_symmetric} that $c(\sigma_{\lambda}^2)=c(\chi_{\lambda}^2)$. By the discussion before \Cref{upperbound_covering_number_sigma_tworow}, we have $\ccn(\sigma_{\lambda};S_n)\geq \lceil \log_2 n \rceil.$ Let $\alpha_0=(n)$. For each $r\geq 1$, let $\alpha_r$ be the partition obtained from the $\alpha_{r-1}\times \lambda$ matrix $A^r$ (by a decreasing rearrangement of its entries) chosen in the following way: If $\alpha_{r-1}=(a_1,a_2,\ldots, a_s)$, then for each $i\geq 1$, choose $A^r_{i,1}=\lceil \frac{a_i}{2} \rceil$ and $A^r_{i,2}=\lfloor \frac{a_i}{2} \rfloor$, until a unique value of $i$, say $x$ (which may or may not exist), from which point the choice of the entries in each row is simply the reverse of the previous choice. The unique value $x$ is obtained in such a  way that the first column sum equals $k$ and the second column sum equals $k-1$. It is clear that when $r=\lceil \log_2 n \rceil-1$, the matrix $A^{r+1}$ obtained by the above choice has all its entries to be 1, and hence the partition obtained from $A^{r+1}$ is $\alpha^{r+1}=(1^n)$. By the discussion prior to \Cref{covering_number_sigma_hook}, we conclude that $\ccn(\sigma_{\lambda};S_n)\leq \lceil \log_2 n \rceil$. Thus, $\ccn(\sigma_{\lambda};S_n)=\lceil \log_2 n \rceil$. Since $n\geq 5$, we get $\ccn(\sigma_{\lambda};S_n)\geq 3$, whence \Cref{basic_lemma_3} implies the result.
	\end{proof}
	
	As an application of \Cref{miller_lemma_1_improved}, we prove that every irreducible character of $S_n$ appears as a constituent of the product of all (non-linear) hook characters.
	
	\begin{theorem}\label{product_of_hooks}
		Let $n\geq 5$. Then $\displaystyle c\left(\prod_{1\leq k\leq n-2}\chi_{\mu(k)}\right)=\irr(S_n)$.
	\end{theorem}

	\begin{proof}
		Let $n\geq 6$ be even. Then, using \Cref{miller_lemma_1_improved}, we get the following: $$\displaystyle c\left(\chi_{(n-1,1)}^2\prod_{2\leq k\leq \frac{n}{2}-1}\chi_{\mu(k)}^2\right)\supseteq c(\sigma_{\mu(2)}\sigma_{\mu(3)}^{\frac{n}{2}-2})=c(\sigma_{\mu(1)}^{\frac{3n}{2}-4})=\irr(S_n),$$ as $\frac{3n}{2}-4\geq n-1$ (see \Cref{covering_number_sigma_hook} and \Cref{basic_lemma_2}). Assume now that $n\geq 9$ is odd. Once again, using \Cref{miller_lemma_1_improved}, we get the following:
		$$c\left(\chi_{(n-1,1)}^2\chi_{\mu(\frac{n-1}{2})}\prod_{2\leq k\leq \frac{n-3}{2}}\chi_{\mu(k)}^2\right)\supseteq c(\sigma_{\mu(2)}\sigma_{\mu(3)}^{\frac{n-5}{2}})=c(\sigma_{\mu(1)}^{\frac{3n-11}{2}})=\irr(S_n),$$
		since $\frac{3n-11}{2}\geq n-1$ (see \Cref{covering_number_sigma_hook} and \Cref{basic_lemma_2}). When $n=5,7$, the result follows by direct computation.
	\end{proof}

	We end this section with the following conjecture:
	
	\begin{conjecture}\label{conjecture_two_row_partition}
		Let $n\geq 5$ and $\lambda=(n-k,k)$ where $1\leq k\leq \lfloor \frac{n}{2} \rfloor$. The following hold.
		\begin{enumerate}
			\item If $n$ is odd, then $c(\chi_{\lambda}^2)=c(\sigma_{\lambda}^2)$. In particular, $\ccn(\chi_{\lambda};S_n)=\ccn(\sigma_{\lambda};S_n)$.
			\item Let $n$ be even, and $\lambda\neq (\frac{n}{2},\frac{n}{2})$. If $1\leq k<\lceil \frac{n}{4} \rceil$, then $c(\chi_{\lambda}^2)=c(\sigma_{\lambda}^2)$. Otherwise,  $c(\chi_{\lambda}^3)=c(\sigma_{\lambda}^3)$. In particular, $\ccn(\chi_{\lambda};S_n)=\ccn(\sigma_{\lambda};S_n)$.
			\item If $n$ is even, $\lceil \log_2 n \rceil\leq \ccn(\chi_{(\frac{n}{2},\frac{n}{2})};S_n)\leq \lceil \log_2n \rceil +1.$
		\end{enumerate}
	\end{conjecture}

	\section{Proof of \Cref{Theorem_3}}
	
	In this section, we prove \Cref{Theorem_3}. Recall that we denote a hook partition $(n-r,1^r)$ by $\mu(r)$. The Durfee rank of a partition $\lambda$ is the size of the largest square contained in it. In particular, partitions with Durfee rank 1 are hook partitions $\mu(r)$, where $0\leq r\leq n-1$. Partitions with Durfee rank 2 are called double hooks. They can be written in the form $(n_4,n_3,2^{d_2},1^{d_1})$. Here we assume $n_4\geq n_3\geq 2$. For  a proposition $P$, the notation $((P))$ takes the value 1 if $P$ is true; otherwise, it takes the value 0. The Kronecker coefficients $g_{\mu\nu\lambda}$ when two of the partitions are hook shapes and the other one is arbitrary were determined by Remmel (\cite[Theorem 2.1]{re1}) and  reproved by Rosas (\cite[Theorem 3]{ro}) (a slightly different statement). We state this result by mixing both to suit our needs. 
	
	\begin{theorem}\label{Kronecker_coeff_two_hook_shapes}
		Let $\mu,\nu,\lambda \vdash n$, where $\mu=(n-e,1^e)$ and $\nu=(n-f,1^f)$ are hook-shaped partitions. Assume that $f\geq e\geq 1$ and $f+e\leq n-1$. Then the Kronecker coefficients $g_{\mu\nu\lambda}$ are given by the following:
		\begin{enumerate}
			\item If $\lambda$ is such that $d(\lambda)>2$, then $g_{\mu\nu\lambda}=0.$
			\item If $\lambda=(n)$, then $g_{\mu\nu\lambda}=1$ if and only if $\mu=\nu$, and if $\lambda=(1^n)$, then $g_{\mu\nu\lambda}=1$ if and only if $\mu=\nu'$.
			\item Let $\lambda=(n-r,1^r)$ be a hook shape where $1\leq r\leq n-2$. Then $$g_{\mu\nu\lambda}=((f-e\leq r\leq e+f+1)).$$
			\item Let $\lambda=(n_4,n_3,2^{d_2},1^{d_1})$ be a double hook where $n_4\geq n_3\geq 2$ and $x=2d_2+d_1$. Then
			\begin{eqnarray*}
				g_{\mu\nu\lambda}&=&((n_3-1\leq \frac{e+f-x}{2}\leq n_4))((f-e\leq d_1))\\ && + ((n_3\leq \frac{e+f-x+1}{2}\leq n_4))((f-e\leq d_1+1)).
			\end{eqnarray*} 
		\end{enumerate}
	\end{theorem}
	
	\noindent The following result of Blasiak will be required. We give a proof as well since it is a direct application of \Cref{Kroneceker_coeff_LR_coeff}.
	
	\begin{lemma}[Proposition 3.1, \cite{bl}] \label{Kroncker_coeff_and_LR_for_hook}Let $\mu(d)$ denote the hook shape $(n-d,1^d)$ where $0\leq d\leq n-1$. Then $\displaystyle g_{\lambda\mu(d)\nu}+g_{\lambda\mu(d-1)\nu}=\sum_{\alpha\vdash d, \beta\vdash n-d}c^{\lambda}_{\alpha\beta}c_{\alpha'\beta}^{\nu}$.
	\end{lemma}

	\begin{proof}
		By \Cref{Kroneceker_coeff_LR_coeff}, we have $\displaystyle s_{\lambda}*s_{(n-d)}s_{(1^d)}=\sum_{\alpha\vdash d, \beta\vdash n-d}c^{\lambda}_{\alpha\beta}s_{\alpha'}s_{\beta}$. On the other hand, using the Pieri rule, it is easy to see that $s_{(n-d)}s_{(1^d)}=s_{\mu(d)}+s_{\mu(d-1)}$. Thus, $$s_{\lambda}*s_{(n-d)}s_{(1^d)}=s_{\lambda}*s_{\mu(d)}+s_{\lambda}*s_{\mu(d-1)}=\sum_{\alpha\vdash d, \beta\vdash n-d}c^{\lambda}_{\alpha\beta}s_{\alpha'}s_{\beta}.$$
		Taking inner-product with $s_{\nu}$ on both sides yields the result.
	\end{proof}
	
	\noindent The following lemma  will be central to the proof of \Cref{Theorem_3}.
	\begin{lemma}\label{hook_containment_lemma}
		Let $1\leq s\leq \frac{n-1}{2}$. Then $c(\chi_{\mu(s-1)}\chi_{\mu(s)})\subseteq c(\chi_{\mu(s)}^2)$ and $c(\chi_{\mu(s-1)}^2)\subseteq c(\chi_{\mu(s)}^2)$. More generally, for any $k\geq 2$, $c(\chi_{\mu(s-1)}\chi_{\mu(s)}^{k-1})\subseteq c(\chi_{\mu(s)}^k)$ and $c(\chi_{\mu(s-1)}^k)\subseteq c(\chi_{\mu(s)}^k)$.
	\end{lemma}

	\begin{proof}
		Note that $\mu(1)=(n-1,1)$, whence the result follows from \Cref{standard_character_kronecker}. So we may assume $s\geq 2$. We show that $c(\chi_{\mu(s)}\chi_{\mu(s-1)})\subseteq c(\chi_{\mu(s)}^2)$. Using \Cref{Kronecker_coeff_two_hook_shapes}(1), we note that if the  Durfee rank of $\lambda$ is greater than 2, then $\chi_{\lambda}$ does not appear as a constituent of both $\chi_{\mu(s)}\chi_{\mu(s-1)}$ and $\chi_{\mu(s)}^2$.
		Further, by our choice of $s$, $\chi_{(1^n)}, \chi_{(n)} \notin c(\chi_{\mu(s-1)\mu(s)})$. Thus, we only consider $\lambda\vdash n$ such that $\lambda=\mu(r)$ where $1\leq r\leq n-2$ or $\lambda$ is a double hook shape.
		
		\medskip
		
		\textbf{Case I}: Suppose $\lambda=(n-r,1^r)$ where $1\leq r\leq n-2$. By \Cref{Kronecker_coeff_two_hook_shapes}(3),  $\chi_{\lambda}\in c(\chi_{\mu(s)}\chi_{\mu(s-1)})$ if and only if $1\leq r\leq 2s$.  The last inequality also implies $0\leq r\leq 2s+1$, whence \Cref{Kronecker_coeff_two_hook_shapes}(3) implies that $\chi_{\lambda}\in c(\chi_{\mu(s)}^2)$.
		
		\medskip
		
		\textbf{Case II}: Assume that $\lambda=(n_4,n_3,2^{d_2},1^{d_1})$, where $n_4\geq n_3\geq 2$ and $x=2d_2+d_1$. Assume that $\chi_{\lambda}\in c(\chi_{\mu(s)}\chi_{\mu(s-1)})$. By \Cref{Kronecker_coeff_two_hook_shapes}(4), the following expression is positive.
		\begin{equation}\label{eqn1}
				((n_3-1\leq \frac{2s-1-x}{2}\leq n_4))((d_1\geq 1))+ ((n_3\leq \frac{2s-x}{2}\leq n_4))((d_1\geq 0)).
		\end{equation}
		To show that $\chi_{\lambda}\in c(\chi_{\mu(s)}^2)$, once again using \Cref{Kronecker_coeff_two_hook_shapes}(4), we need to show that the following expression is positive.
		\begin{equation}\label{eqn2}
			((n_3-1\leq \frac{2s-x}{2}\leq n_4))((d_1\geq 0))+ ((n_3\leq \frac{2s-x+1}{2}\leq n_4))((d_1+1\geq 0)).
		\end{equation}
		Notice that if the second summand in expression \ref{eqn1} is positive, then the first summand in expression \ref{eqn2} is positive as well, and we are done. Therefore, we may assume that the second summand in expression \ref{eqn1} is zero. In that case, the first summand in expression \ref{eqn1} is positive. This automatically means that $d_1\geq 1$. We also have $n_3-\frac{1}{2}\leq \frac{2s-x}{2}\leq n_4+\frac{1}{2}$. But since the second summand  in expression \ref{eqn1} is zero, it follows that either $\frac{2s-x}{2}$ takes the value $n_3-\frac{1}{2}$ or $n_4+\frac{1}{2}$. In the former case, the first summand in expression \ref{eqn2} is positive, and hence we are done. On the other hand, the latter case is not possible. To see this, we note that $n=n_4+n_3+x=(n_4+\frac{x}{2})+(n_3+\frac{x}{2})$. Since $n_4\geq n_3$, we conclude that $n_3+\frac{x}{2}\leq \frac{n}{2}$. Also, $n_4=n-n_3-x$  implies that $\frac{2s-x}{2}-\frac{1}{2}=n-n_3-x$, whence 	$n+\frac{1}{2}=s+n_3+\frac{x}{2}\leq n$ as $s\leq \frac{n}{2}$. We get a contradiction, as desired.	
		
		\medskip
		
		Similar computations yield that $c(\chi_{\mu(s-1)}^2)\subseteq c(\chi_{\mu(s)}^2)$. The latter statements follow immediately by using  \Cref{basic_lemma_1}.
	\end{proof}
	
	We can now prove the first part of \Cref{Theorem_3}.
	
	\begin{proof}[\textbf{Proof of \Cref{Theorem_3}(1)}] 
		Since $\ccn(\chi_{\lambda};S_n)=\ccn(\chi_{\lambda'};S_n)$, it is enough to prove the result for $\chi_{\mu(2)}$. The result can be easily verified for $5\leq n\leq 9$ using SageMath (\cite{sam}). So we may assume $n\geq 10$.
		
		Using \Cref{Kronecker_coeff_two_hook_shapes}, we get that $c(\chi_{\mu(2)}^2)=\{\chi_{\lambda}\mid \lambda_1=n-4+i, \lambda_2\leq 2+i; 0\leq i\leq 4\}$. We claim the following: for any $2\leq k\leq \lfloor\frac{n}{3}\rfloor$, 
		\begin{equation}\label{claim_1_Theorem3_1}
				c(\chi_{\mu(2)}^k)\supseteq \{\chi_{\lambda}\mid \lambda_1=n-2k+i, \lambda_2\leq k+i;0\leq i \leq 2k\}.
		\end{equation}
		We make two observations at this point: (a) $2\leq k\leq \lfloor\frac{n}{3} \rfloor$ ensures that $n-2k+i\geq k+i$ for every $0\leq i\leq 2k$, (b) since $\chi_{\mu(2)}\in c(\chi_{\mu(2)}^2)$, we get that $c(\chi_{\mu(2)}^2)\subseteq c(\chi_{\mu(2)}^3)\subseteq c(\chi_{\mu(2)}^4)\subseteq \dots$. We now prove our claim by induction on $k$. The claim holds for $k=2$. Thus, we may assume $k\geq 3$. Since $c(\chi_{\mu(2)}^{k-1})\subseteq c(\chi_{\mu(2)}^k)$ by observation (b) made above, we conclude by induction hypothesis that 
		\begin{equation}\label{induction_hypothesis}
			\{\chi_{\lambda} \mid \lambda_1=n-2k+2+i,\lambda_2\leq k+i-1; 0\leq i\leq 2k-2\}\subseteq c(\chi_{\mu(2)}^k).
		\end{equation}
		Thus, it remains to show that $\chi_{\lambda}\in c(\chi_{\mu(2)}^k)$ whenever  $\lambda$ satisfies one of the following conditions: (1) $\lambda_1=n-2k+i, \lambda_2\leq k+i$ where $i=0,1$ and (2) $\lambda_1=n-2k+i$, $k+i-2\leq \lambda_2\leq k+i$ for every $2\leq i\leq 2k$. In the latter case, we may assume that $i\leq \frac{k+2}{2}$ since if $i>\frac{k+2}{2}$, then $\lambda_2\geq k+i-2$ is not possible as $\lambda_2\leq 2k-i$. Now we determine the conditions under which $\lambda$ is a two-row partition, that is, $\lambda_3=0$. In (1), clearly $\lambda_3>0$. In (2), if $i<\frac{k}{2}$, then $\lambda_3>0$. Thus, $\lambda_3>0$ unless (a) $i=\frac{k+1}{2}$ where $k$ is odd, whence $\lambda=(n-\frac{3k-1}{2},\frac{3k-1}{2})$ and (b)  $\frac{k}{2}\leq i\leq \frac{k}{2}+1$ where $k$ is even, whence $\lambda=(n-\frac{3k}{2},\frac{3k}{2})$, or $\lambda=(n-\frac{3k}{2}+1,\frac{3k}{2}-1)$.
		
		\medskip

		\textbf{Case I}: Let $\lambda=(\lambda_1,\lambda_2,\ldots)$ be as in (1) or (2) with  $\lambda_3>0$. Let $\tilde{\lambda}=(\lambda_2,\lambda_3,\ldots)$ so that $|\tilde{\lambda}|=2k-i$. Since $\lambda_3>0$, there exists $\tilde{\beta}\vdash 2k-2-i$ such that $\tilde{\lambda}/\tilde{\beta}$ is one of the two shapes \begin{ytableau}
			\none &   \\
			& \none
		\end{ytableau}\;\;\; or, \;
		\begin{ytableau}
			\\
			\\
		\end{ytableau}. We make the choice of $\tilde{\beta}$ arbitrary with the above property provided $\lambda_2\leq k+i-1$. If $\lambda_2=k+i$, notice that $\lambda_3\leq k+i$ and $\lambda_3=k+i$ only when $i=0$, in which case $\lambda=(n-2k,k,k)$. Either way, we choose $\tilde{\beta}$ by removing one box from the first row of $\tilde{\lambda}$ and another box from a different admissible row of $\tilde{\lambda}$. Thus, we obtain that $\tilde{\beta}_1\leq k-1+i$ in all cases. We define $\beta=(\beta_1,\beta_2,\ldots )\vdash n-2$ as follows:
		\[\beta_r=
		\begin{cases*}
			\lambda_1 & \text{if $r=1$}\\
			\tilde{\beta}_{r-1} & \text{if $r\geq 2$}
		\end{cases*}.
		\]
	 	Let $\alpha=(1,1)$. Clearly $c^{\lambda}_{\alpha\beta}=1$. Now we define a partition $\nu=(\nu_1,\nu_2,\ldots)\vdash n$ as follows:
		\[\nu_r=
		\begin{cases*}
			n-2k+2+i & \text{if $r=1$}\\
			\beta_r & \text{if $r\geq 2$}
		\end{cases*}.
		\]
		Clearly, $\nu_2\leq k-1+i$ since $\beta_2=\tilde{\beta}_1\leq k-1+i$. By \Cref{induction_hypothesis}, $\chi_{\nu}\in c(\chi_{\mu(2)}^{k-1})$. Further, $\nu/\beta$ is a horizontal strip with two boxes, which implies that  $c^{\nu}_{\alpha'\beta}=1$. Thus, we can conclude that $\displaystyle \sum_{\gamma\vdash 2,\delta\vdash n-2} c^{\lambda}_{\gamma\delta}c^{\nu}_{\gamma'\delta}$ is positive, whence \Cref{Kroncker_coeff_and_LR_for_hook} implies $g_{\lambda\mu(2)\nu}+g_{\lambda\mu(1)\nu}>0$. If $g_{\lambda\mu(2)\nu}>0$, then $\chi_{\lambda}\in c(\chi_{\nu}\chi_{\mu(2)})\subseteq c(\chi_{\mu(2)}^k)$ and we are done. If $g_{\lambda\mu(1)\nu}>0$, then $\chi_{\lambda}\in c(\chi_{\nu}\chi_{\mu(1)})\subseteq c(\chi_{\mu(2)}^{k-1}\chi_{\mu(1)})$. By \Cref{hook_containment_lemma}, $c(\chi_{\mu(1)}\chi_{\mu(2)}^{k-1})\subseteq c(\chi_{\mu(2)}^k)$, whence we conclude that $\chi_{\lambda}\in c(\chi_{\mu(2)}^k)$ and we are done.

		By the assumption in \textbf{Case I} and the discussion prior to it, we are left with the following case:
		
		\medskip
		
		\noindent \textbf{Case II}: (a) $\lambda=(n-\frac{3k-1}{2}, \frac{3k-1}{2})$ when $k$ is odd and (b) $\lambda=(n-\frac{3k}{2},\frac{3k}{2})$, or $\lambda=(n-\frac{3k}{2}+1,\frac{3k}{2}-1)$ when $k$ is even. First, we consider (a).
		Since $k\geq 3$, $\lambda_2\geq 4$. Let $\alpha=(2)\vdash 2$. Let $\beta=(n-\frac{3k-1}{2}, \frac{3k-1}{2}-2)$. Clearly $c^{\lambda}_{\alpha\beta}=1$. Let $\nu=(n-\frac{3k-1}{2}+1, \frac{3k-1}{2}-2,1)$. It is immediate that $c^{\nu}_{\alpha'\beta}=1$ as $\nu/\beta$ is a vertical strip with two boxes. Thus, we obtain $\displaystyle \sum_{\gamma\vdash 2,\delta\vdash n-2} c^{\lambda}_{\gamma\delta}c^{\nu}_{\gamma'\delta}>0$ . Notice that $\chi_{\nu}\in c(\chi_{\mu(2)}^{k-1})$.  By \Cref{Kroncker_coeff_and_LR_for_hook}, we get $g_{\lambda\mu(2)\nu}+g_{\lambda\mu(1)\nu}>0$. Suppose that $g_{\lambda\mu(2)\nu}>0$. This implies $\chi_{\lambda}\in c(\chi_{\nu}\chi_{\mu(2)})\subseteq c(\chi_{\mu(2)}^k)$ and we are done. If $g_{\lambda\mu(1)\nu}>0$, then $\chi_{\lambda}\in c(\chi_{\nu}\chi_{\mu(1)})\subseteq c(\chi_{\mu(2)}^{k-1}\chi_{\mu(1)})$. By \Cref{hook_containment_lemma}, $c(\chi_{\mu(1)}\chi_{\mu(2)}^{k-1})\subseteq c(\chi_{\mu(2)}^k)$, whence we conclude that $\chi_{\lambda}\in c(\chi_{\mu(2)}^k)$ and we are done. The arguments for the cases in (b) follow along similar lines. We simply mention the choices of $\alpha,\beta$, and $\nu$ in these cases. If $\lambda=(n-\frac{3k}{2}, \frac{3k}{2})$, then $\lambda_2\geq 5$ as $k\geq 3$. We choose $\alpha=(2)\vdash 2$, $\beta=(n-\frac{3k}{2},\frac{3k}{2}-2)\vdash n-2$, and $\nu=(n-\frac{3k}{2}+1,\frac{3k}{2}-2,1)\vdash n$. If $\lambda=(n-\frac{3k}{2}+1, \frac{3k}{2}-1)$, then $\lambda_2\geq 4$ as $k\geq 3$. We choose $\alpha=(2)\vdash 2$, $\beta=(n-\frac{3k}{2}+1,\frac{3k}{2}-3)\vdash n-2$, and $\nu=(n-\frac{3k}{2}+2,\frac{3k}{2}-3,1)\vdash n$.
		
		\medskip
		
		By \Cref{hook_constituents} we get,
		
		\begin{equation}\label{auxillary_eqn}
			c(\sigma_{\mu(2)}^k)=\{\chi_{\lambda}\mid \lambda_1\geq n-2k\}=\{\chi_{\lambda}\mid \lambda_1=n-2k+i; 0\leq i\leq 2k\}.
		\end{equation}
		Now we compare the sets in the RHS of \Cref{claim_1_Theorem3_1} and \Cref{auxillary_eqn} by taking $k=\lfloor\frac{n}{3}\rfloor$. It is easy to check that when $n\equiv 0(\text{mod}\; 3)$, that is, $n=3k$, then both sets are equal, whence we conclude that $c(\chi_{\mu(2)}^k)=c(\sigma_{\mu(2)}^k)$. Assume now that $n=3k+1$. In this case, it is easy to see that $c(\sigma_{\mu(2)}^k)\setminus c(\chi_{\mu(2)}^k)=\{\chi_{\lambda} \mid \lambda_1=\lambda_2=n-2k+i; 0\leq i\leq \frac{k-1}{2}\}.$ We claim that $c(\sigma_{\mu(2)}^{k+1})=c(\chi_{\mu(2)}^{k+1})$. Since $c(\sigma_{\mu(2)}^{k+1})\supseteq c(\chi_{\mu(2)}^{k+1})$, using \Cref{hook_constituents} and the fact that $c(\chi_{\mu(2)}^{k})\subseteq c(\chi_{\mu(2)}^{k+1})$, we need to show that $\chi_{\lambda}\in c(\chi_{\mu(2)}^{k+1})$, where  $\lambda$ satisfies one of the following conditions: (1) $\lambda_1=n-2k-i$ where $i=1,2$ and (2) $\lambda_1=\lambda_2=n-2k+i$ for every $0\leq i\leq \frac{k-1}{2}$. To show this, we repeat the same process as done previously in \textbf{Case I} and \textbf{Case II}. More specifically, we note that $\lambda_3>0$ in both (1) and (2) except when $k$ is odd and $\lambda=(n-2k+i,n-2k+i)$ with $i=\frac{k-1}{2}$. Thus, when $\lambda_3>0$, we use the argument in \textbf{Case I} above, and for the latter case, the argument in \textbf{Case II} is used. When $n=3k+2$, it can be proved similarly that $c(\chi_{\mu(2)}^{k+1})=c(\sigma_{\mu(2)}^{k+1})$. Overall, we conclude that $c(\sigma_{\mu(2)}^{\lceil \frac{n}{3} \rceil})=c(\chi_{\mu(2)}^{\lceil \frac{n}{3} \rceil})$. The proof of the theorem now follows from \Cref{basic_lemma_3} and \Cref{covering_number_sigma_hook}.
	\end{proof}
	
	\begin{remark}
		In the above proof, one can prove that \Cref{claim_1_Theorem3_1} can be strengthened to an equality.
	\end{remark}
		
	We now move on to the proof of the second statement of \Cref{Theorem_3}. It is enough to consider $\lambda=(\frac{n+1}{2},1^{\frac{n-1}{2}})$ when $n$ is odd, and $\lambda=(\frac{n}{2}+1, 1^{\frac{n}{2}-1})$ when $n$ is even. In other words, $\lambda=\mu(k)$, where $k=\lfloor \frac{n-1}{2} \rfloor$. We will prove our theorem via a sequence of lemmas.
	
	\begin{lemma}\label{durfee_lemma_1}
		Let  $k=\lfloor \frac{n-1}{2} \rfloor$ and $\lambda\vdash n$ be such that $d(\lambda)=2m$, where $m\geq 2$. Then there exists $\alpha\vdash k$ and $\beta \vdash n-k$ such that $d(\alpha)=d(\beta)=m$ and  $c^{\lambda}_{\alpha\beta}>0$ (i.e., $\langle s_{\alpha}s_{\beta},s_{\lambda} \rangle>0)$.
	\end{lemma}
	
	\begin{proof}
		At first, we define a partition $\beta\subseteq \lambda$ by constructing its Young diagram $T_{\beta}$ from that of $T_{\lambda}$ as follows:
		\begin{itemize}
			\item We choose a box from the first $m$ rows of $T_{\lambda}$ if there are at least $m$ boxes below it in its column.
			
			\vspace{2 mm}
			
			\item We say that a box in $T_{\lambda}$ satisfy \textbf{Property P} if there are at least $m$ boxes right to it (in its row). For each $1\leq j\leq m$, we choose those boxes in the $j$-th column that satisfy \textbf{Property P} but aren't  among the last $m$ boxes in the same column that satisfy \textbf{Property P}.
			
			\vspace{2 mm}
			
			\item At this point, we observe that the number of boxes chosen is clearly less than or equal to $n-k$, since there is an obvious injection from the set of chosen boxes to the set of non-chosen ones.
			
			\vspace{2 mm}
			
			\item Now if there are more boxes to choose for $T_{\beta}$, we do it by choosing boxes from the first column (column-wise), and then from the second column, and so on till the $m$-th column.
			
			\vspace{2 mm}
			
			\item If even more boxes are required, we choose boxes from the first row (row-wise), then from the second row, and so on until the $m$-th row.
			
			\vspace{2 mm} 
			
			\item In the above steps, we stop at the point where we have chosen $n-k$ many boxes. Since $d(\lambda)=2m$, the total number of boxes that are neither in the first $m$ rows nor in the first $m$ columns is clearly less than or equal to $k$, and hence, by performing the steps mentioned above, we certainly get $n-k$ many boxes at some stage.
			
			\vspace{2 mm}
		
			\item Finally, it is also clear that $d(\beta)=m$.
		\end{itemize}
		\noindent We give an example to illustrate the choice of $\beta$ (see \Cref{Fig_6}). Let $\lambda=(11,10^2,8,7,6^2,4^2,2^3,1)\vdash 73$. We have $d(\lambda)=6$, $m=3$, and  $k=36$. Thus, $\beta\vdash 37$.  The red colored boxes in the first diagram of the given illustration are the ones chosen in steps 1 and 2. The second diagram is obtained from the first by performing the steps after step 2 (indicated by yellow colored boxes) and finally gives $\beta=(8,7,6,3,2^4,1^5)\vdash 37$.
	
		\begin{figure}[!h]
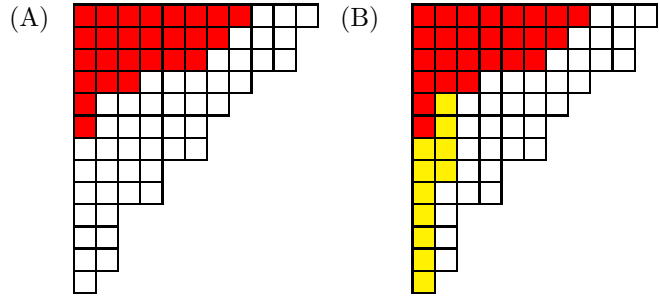

			(A)\;\;\ytableausetup{smalltableaux}
				\begin{ytableau}
				*(red) & *(red) & *(red) & *(red) & *(red) & *(red) & *(red) & *(red) & & &\\
				*(red) & *(red) & *(red) & *(red) & *(red) & *(red) & *(red) & & &\\
				*(red) & *(red) & *(red) & *(red) & *(red) & *(red) & & & &\\
				*(red) &  *(red) & *(red) & & & & &\\
				*(red) & & & & & &\\
				*(red) & & & & &\\
				& & & & &\\
				& & &\\
				& & &\\
				& \\
				&\\
				&\\
				\\
			\end{ytableau}\;\;\;(B)\;\;
			\ytableausetup{smalltableaux}
			\begin{ytableau}
				*(red) & *(red) & *(red) & *(red) & *(red) & *(red) & *(red) & *(red) & & &\\
				*(red) & *(red) & *(red) & *(red) & *(red) & *(red) & *(red) & & &\\
				*(red) & *(red) & *(red) & *(red) & *(red) & *(red) & & & &\\
				*(red) &  *(red) & *(red) & & & & &\\
				*(red) & *(yellow) & & & & &\\
				*(red) & *(yellow) & & & &\\
				*(yellow) & *(yellow) & & & &\\
				*(yellow)& *(yellow) & &\\
				*(yellow)& & &\\
				*(yellow)& \\
				*(yellow)&\\
				*(yellow)&\\
				*(yellow)\\
			\end{ytableau}
			\caption{Construction of $\beta\subseteq \lambda$.}
			\label{Fig_6}
		\end{figure}

		\medskip
	
		Now it is required to fill the skew shape $T_{\lambda/\beta}$ in a way that its type $\alpha$ has Durfee rank equal to $m$, and it is a LR tableaux. For that, we make some observations on the skew diagram $T_{\lambda/\beta}$. It is easy to see that $T_{\lambda/\beta}$ has no $(m+1)\times (m+1)$ square contained in it. We use the convention that the rows of $T_{\lambda/\beta}$ are labeled with respect to those of $T_{\lambda}$. With this convention, let $p$ be the least positive integer such that the number of boxes (say $r_0$) in the $(2m+p)$-th row of $T_{\lambda/\beta}$ is less than $m$, that is, $r_0<m$. Note that $p$ always exists with the convention that there are zeroes after the last part of the skew partition $\lambda/\beta$. Let $\widetilde{T_{\lambda/\beta}}$ be the sub-diagram of $T_{\lambda/\beta}$ starting from the $(2m+p)$-th row to the last row (say $(2m+p+d)$-th row). We write $\widetilde{T_{\lambda/\beta}}=(r_0,r_1,\ldots,r_d)$, where $r_i$ denotes the number of boxes in the $(2m+p+i)$-th row.  Now we make an important observation on $\widetilde{T_{\lambda/\beta}}$. There exists a positive integer $c$ (possibly) such that $r_0\geq \cdots \geq r_{c-1}$, $r_c=r_{c-1}+1$, and $r_c\geq \cdots \geq r_d$. We remark that such a \emph{`c'} may not exist, in which case $r_0\geq \cdots \geq r_d$. To see this, consider the box (if there is any) of $T_{\lambda}$ that lies to the immediate left of the first box of the $(2m+p)$-th row of  $T_{\lambda/\beta}$. If there is no such box, then it is clear that such a \emph{`c'} won't exist. But if such a box exists, then (a) it is a part of the diagram $T_{\beta}$ and hence it lies in some $k$-th column of $T_{\lambda}$ where $1\leq k\leq m$, (b) in the $k$-th column, our concerned box lies just below the bottom most box of that column having \textbf{Property P}. By our choice of $\beta$, all the boxes from the first column to the $(k-1)$-th column are in the diagram of $T_{\beta}$, whence our observation follows.
		
		\medskip
		
		\noindent With these observations, we are now in a position to fill $T_{\lambda/\beta}$. We divide this filling into three parts.
		
		\begin{itemize}
			\item For the first $2m$ rows of $T_{\lambda/\beta}$, we fill each column with numbers $1$ to $j$ in increasing order, where $j$ is the number of boxes in that column. Clearly, $j\leq m$ by our choice of $\beta$. Since $d(\lambda)=2m$, it is also clear that each of the numbers from 1 to $m$ has been used at least $m$ times. Observe also that the filling until the $2m$-th row is semi-standard and the reverse reading word  is a lattice permutation.
			
			\vspace{2 mm}
			
			\item  Now we fill  the sub-diagram  of $T_{\lambda/\beta}$ from $(2m+1)$-th to $(2m+p-1)$-th row (at this step it is assumed $p\geq 2$). Let this sub-diagram be written as $(s_1,\ldots,s_{p-1})$, where $s_i$ is the number of boxes in the $(2m+i)$-th row. Then $m\leq s_i \leq 2m$ (by the choice of $p$). We fill the last $m$ boxes of the $(2m+i)$-th row by the number $m+i$, and then we fill the row from left to right with the least possible numbers in a semi-standard way. Note that these \emph{``least possible numbers''} will always be less than or equal to $m$. If not, then it implies that the diagram would contain a $(m+1)\times (m+1)$ square, a contradiction. Thus, we have a filling until the $(2m+p-1)$-th row, which still has the desired properties of a LR tableaux. Further, any number greater than $m$ that has been used occurs exactly $m$ many times.
			
			\vspace{2 mm}
			
			\item Finally, we fill last remaining sub-diagram $\widetilde{T_{\lambda/\beta}}$, which is from $(2m+p)$-th row to $(2m+p+d)$-th row. Recall the observation we made about this sub-diagram before. If $r_0\geq \cdots \geq r_d$, we fill all the boxes of the $(2m+p+i)$-th row with $m+p+i$, where $0\leq i\leq d$. If not, then $r_0\geq \cdots \geq r_{c-1}$, $r_c=r_{c-1}+1$, and $r_c\geq\cdots\geq r_d$. Let $e\geq 2$ be the least positive integer such that $r_{c+e}\leq r_0$. If such $e$ does not exist, we take $e=d+1$. Then $r_{c-1}+1=r_c=\cdots=r_{c+e-1}$. We fill all boxes of the $(2m+p+i)$-th row with $m+p+i$, where $1\leq i\leq c-1$. For $c\leq i\leq e-1$, we fill the first box of the $(2m+p+i)$-th row with $m+p+i-c$ and the remaining boxes with $m+p+i$. For $e\leq i\leq d$, all boxes of the $(2m+p+i)$-th row are filled with $m+p+i$ once again.
		\end{itemize}
		
		\medskip
		
		\noindent Overall, we get a LR tableaux whose type has Durfee rank $m$. We end our proof by filling the tableaux in our example according to the method described above (see \Cref{Fig_7} below). In this case, $\alpha=(9,8^2,3^2,2,1^3)$.
		
		\begin{figure}[!h]
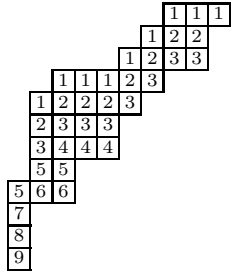

			\ytableausetup{centertableaux}
			\begin{ytableau}
				\none & \none & \none & \none & \none & \none & \none & \none & 1 & 1 & 1\\
				\none & \none & \none & \none & \none & \none & \none & 1 & 2 & 2\\
				\none & \none & \none & \none & \none & \none & 1 & 2 & 3 & 3\\
				\none & \none  & \none & 1 & 1 & 1 & 2 & 3\\
				\none & \none & 1 & 2 & 2 & 2 & 3\\
				\none & \none & 2 & 3 & 3 & 3\\
				\none & \none & 3 & 4 & 4 & 4\\
				\none & \none & 5 & 5\\
				\none & 5 & 6 & 6\\
				\none& 7 \\
				\none& 8 \\
				\none& 9 
			\end{ytableau}
			\hspace{-2 cm}\caption{LR tableaux of shape $(11,10^2,8,7,6^2,4^2,2^3,1)/(8,7,6,3,2^4,1^5)$ and type $(9,8^2,3^2,2,1^3)$.}
			\label{Fig_7}
		\end{figure}
	\end{proof}

	We prove an analogue of \Cref{durfee_lemma_1} when $d(\lambda)$ is odd. The proof is similar, but there are subtle changes, so we write it in detail.

	\begin{lemma}\label{durfee_lemma_2}
		Let  $k=\lfloor \frac{n-1}{2} \rfloor$ and $\lambda\vdash n$ be such that $d(\lambda)=2m-1$, where $m\geq 2$. Then there exists $\alpha\vdash k$ and $\beta \vdash n-k$ such that $m-1\leq d(\alpha)\leq m$, $d(\beta)=m$, and  $c^{\lambda}_{\alpha\beta}>0$ (i.e., $\langle s_{\alpha}s_{\beta},s_{\lambda} \rangle>0)$.
	\end{lemma}

	\begin{proof}
		As before, we define the partition $\beta\subseteq \lambda$ by constructing its Young diagram $T_{\beta}$ from that of $T_{\lambda}$ as follows:
		\begin{itemize}
			\item We choose all boxes from the first $m\times m$ square of $T_{\lambda}$.
			
			\vspace{2 mm}
			
			\item We choose a box from the first $(m-1)$ rows of $T_{\lambda}$ if there are at least $m$ boxes below it in its column.
			
			\vspace{2 mm}
			
			\item We say that a box in $T_{\lambda}$ satisfy \textbf{Property P} if there are at least $m$ boxes right to it (in its row). For each $1\leq j\leq m-1$, we choose those boxes in the $j$-th column that satisfy \textbf{Property P} but aren't  among the last $(m-1)$ boxes in the same column that satisfy \textbf{Property P}.
			
			\vspace{2 mm}
			
			\item At this point, we observe that the number of boxes chosen is clearly less than or equal to $n-k$, since there is an obvious injection from the set of chosen boxes to the set of non-chosen ones.
			
			\vspace{2 mm}
			
			\item Now, if there are more boxes to choose to form $\beta$, we do it in the following sequence:
				
				\vspace{2 mm}
				
				\begin{enumerate}
					\item Choose boxes from the first column (column-wise), then from the second column, and so on until the $(m-1)$-th column.
					
					\vspace{2 mm}
					
					\item Choose boxes from the first row (row-wise), then from the second row, and so on until the $(m-1)$-th row.
					
					\vspace{2 mm}
					
					\item Choose boxes from the $m$-th column (column-wise).
					
					\vspace{2 mm}
					
					\item Choose boxes from the $m$-th row (row-wise).
				\end{enumerate}
				
				\vspace{2 mm}
				
				\noindent We stop at that point of the above sequence when the total number of boxes chosen equals $n-k$.
				
			\vspace{2 mm}
				
			\item Since $d(\lambda)=2m-1$, the total number of boxes that are neither in the first $m$ rows nor in the first $m$ columns is clearly less than or equal to $k$, and hence, by performing the steps mentioned above, we certainly get $n-k$ many boxes at some stage.
			
			\vspace{2 mm}
			
			\item Finally, it is also clear that $d(\beta)=m$.
		\end{itemize}
		Now it is required to fill the skew shape $T_{\lambda/\beta}$ in a way that its type $\alpha$ has Durfee rank either $m-1$ or $m$ and it is a LR tableaux. For that, we make some observations on the skew diagram $T_{\lambda/\beta}$. We use the convention that the rows of $T_{\lambda/\beta}$ are labeled with respect to those of $T_{\lambda}$. With this convention, we have the following observations:
		
		\begin{itemize}
			\item The sub-diagram of $T_{\lambda/\beta}$ from the first row to the $(2m-1)$-th row has a $m\times m-1$ rectangle contained in it. To see this, write $\lambda=(\lambda_1,\ldots,\lambda_{2m-1},\lambda_{2m},\ldots, \lambda_{l})$ where $\lambda_i\geq 2m-1$ if $1\leq i\leq 2m-1$, and $\lambda_i\leq 2m-1$ if $2m\leq i\leq p$. Let $\beta=(\beta_1,\beta_2,\ldots)$. Now assume that our assertion is not true. By our choice of $\beta$, we have the following inequalities: (a) $\beta_i=\lambda_i$ when $1\leq i\leq m-1$, (b) $\beta_m\geq m$, $\beta_i=m$ when $m+1\leq i\leq 2m-1$, and there exists $j\in\{m,m+1,\ldots,2m-1\}$ such that $\lambda_j<\beta_m+m-1$, and (c) $\lambda_i-\beta_i<\beta_i$ for all $2m\leq i\leq p$. Note that
			\begin{equation}
				|T_{\beta}|-|T_{\lambda/\beta}|=\sum_{i}\beta_i-\sum_{i}(\lambda_i-\beta_i)=\sum_{i}(2\beta_i-\lambda_i)\leq 2,
			\end{equation}
			since $\beta\vdash n-k$ and $\alpha\vdash k$. Now, 
			$$\sum_i (2\beta_i-\lambda_i)=\sum_{i=1}^{m-1}\lambda_i+(2\beta_m-\lambda_m)+\sum_{i=m+1}^{2m-1} (2m-\lambda_i)+\sum_{i=2m}^{p}(2\beta_i-\lambda_i).$$
			The last summand in the above equation is always non-negative. In (b), if $j=m$, that is, $\lambda_m<\beta_m+m-1$, then $2\beta_m-\lambda_m>\beta_m+1-m\geq 1$. Also, $2m+\lambda_i-\lambda_{m+i}\geq 2m$ for all $1\leq i\leq m-1$. Hence $\sum_{i}(2\beta_i-\lambda_i)> 2$, a contradiction. Otherwise, in (b), if $j\neq m$, then $2m-\lambda_j>m-\beta_m+1$, whence $(2\beta_m-\lambda_m)+(2m-\lambda_j)>\beta_m+m+1-\lambda_m\geq 2m+1-\lambda_m$. But $\lambda_1+2m+1-\lambda_m\geq 2m+1$ and $\lambda_i+2m-\lambda_{k}\geq 2m$ for distinct pairs $(i,k)$ such that $2\leq i\leq m-1$ and $m+1\leq k\leq 2m-1$ with  $k\neq j$. This yields $\sum_{i}(2\beta_i-\lambda_i)>2$, once again a contradiction.
			
			\vspace{2 mm}
			
			\item There is no $(m+1) \times (m+1)$ square contained in $T_{\lambda/\beta}$.
			
			\vspace{2 mm}
			
			\item Let $p$ be the least positive integer such that the number of boxes (say $r_0$) in the $(2m-1+p)$-th row of $T_{\lambda/\beta}$ is less than $m$, that is, $r_0<m$. Note that $p$ always exists with the convention that there are zeroes after the last part of the skew partition $\lambda/\beta$. Let $\widetilde{T_{\lambda/\beta}}$ be the sub-diagram of $T_{\lambda/\beta}$ starting from the $(2m-1+p)$-th row to the last row (say $(2m-1+p+d)$-th row). We write $\widetilde{T_{\lambda/\beta}}=(r_0,r_1,\ldots,r_d)$, where $r_i$ denotes the number of boxes in the $(2m-1+p+i)$-th row. Then there exists a positive integer $c$ (possibly) such that $r_0\geq \cdots \geq r_{c-1}$, $r_c=r_{c-1}+1$, and $r_c\geq \cdots \geq r_d$. We remark that such a $c$ may not exist, in which case $r_0\geq \cdots \geq r_d$. The reason for this is exactly as in the previous lemma.
		\end{itemize}
		
		\medskip
		
		\noindent With these observations, we are now in a position to fill $T_{\lambda/\beta}$. We divide this filling into three parts.
		
		\medskip
		
		For the first $(2m-1)$ rows of $T_{\lambda/\beta}$, we fill each column with numbers $1$ to $j$ in increasing order, where $j$ is the number of boxes in that column. Clearly, $j\leq m$ by our choice of $\beta$. Also, since there is a $m\times m-1$ rectangle in the first $2m-1$ rows, we conclude that each of the numbers from $1$ to $m$ has been used at least $m-1$ many times. It is obvious that the filling till the  $(2m-1)$-th row is semi-standard, and the reverse reading word  is a lattice permutation. To fill the remaining part of the diagram, we make cases.
			
		\noindent \textbf{Case I}: Suppose that in our filling of the first $(2m-1)$ rows, the number $m$ (hence all the numbers from $1$ to $m-1$) has occurred at least $m$ times. In this case, the next two sub-diagrams are to be filled in exactly the same way as we did in the previous lemma. We observe that the type $\alpha$ has Durfee rank $m$ in this case.
		
		\medskip
		
		\noindent \textbf{Case II}: Suppose that in our filling of the first $(2m-1)$ rows, the number $m$ has occurred exactly $m-1$ many times. Now we fill  the sub-diagram  of $T_{\lambda/\beta}$ from $2m$-th to $(2m-2+p)$-th row (at this step it is assumed $p\geq 2$). Let this sub-diagram be written as $(s_1,\ldots,s_{p-1})$, where $s_i$ is the number of boxes in the $(2m-1+i)$-th row. Then $m\leq s_i \leq 2m-1$ (by the choice of $p$). We fill the boxes of the $(2m-1+i)$-th row by the following two steps.
		
		\vspace{2 mm}
		
		\noindent \textbf{Step 1}: We fill the last $m-1$ boxes of the row by  the number $m+i$.
		
		\vspace{2 mm}
		
		\noindent \textbf{Step 2}: We fill the row from left to right with \emph{``least possible numbers''} in a semi-standard way, but with the restriction that whenever a number bigger than $m$ is required it must be the least number that has not been used $m$ many times (in the filling till then).
		
		\medskip
		
		\noindent The reason we can get this done in a semi-standard way is as follows: For any $1\leq i\leq p-1$, consider the $(2m-1+i)$-th row. Notice that if any number greater than $m$ appears in this row, it must appear exactly once. Indeed, if it appears twice, then there is a $(m+1)\times (m+1)$ square contained in the diagram $T_{\lambda/\beta}$, a contradiction. Suppose now that this unique number bigger than $m$ has appeared in the concerned row. We claim that this number is certainly less than or equal to $m+i$. If not, then by our process mentioned in step 2, we get that there is a box in the skew diagram of $T_{\lambda/\beta}$ that lies in the first $m\times m$ square of $T_{\lambda}$. This is once again a contradiction to the choice of $\beta$. Thus, we conclude that our filing in step 2 is possible (that is, our filling is semi-standard). Also, the fact that the reverse reading word (till now) is a lattice permutation follows from the choice of our filling. 
		
		\medskip
		
		Now we proceed to fill the last diagram $\widetilde{T_{\lambda/\beta}}$ (at this point, recall the observation we made about this diagram before). Let $t$ be the least positive integer such that for each $t\leq j\leq p-1$, $m+j$ has occurred exactly $m-1$ many times (note that these occur in the previous sub-diagram). We need to fill from the $(2m-1+p)$-th row to the $(2m-1+p+d)$-th row.  If $r_0\geq \cdots \geq r_d$, we fill all the boxes of the $(2m-1+p+i)$-th row with $m+p+i$ where $0\leq i\leq d$. Otherwise, $r_0\geq \cdots \geq r_{c-1}$, $r_c=r_{c-1}+1$, and $r_c\geq \cdots \geq r_d$. Let $e\geq 2$ be the least positive integer such that $r_{c+e}\leq r_0$. If such $e$ does not exist, we take $e=d+1$. Then $r_{c-1}+1=r_c=\cdots=r_{c+e-1}$. We fill all boxes of the $(2m+p+i)$-th row with $m+p+i$ when $1\leq i\leq c-1$. For $c\leq i\leq e-1$, we fill the first box of the $(2m+p+i)$-th row with the least number possible except the numbers $m+1$ to $m+t-1$ (thus if $t=1$ we have no restriction), and the remaining boxes with $m+p+i$. For $e\leq i\leq d$, all boxes of the $(2m+p+i)$-th row are filled with $m+p+i$ once again.

		\medskip
		
		\noindent Overall, we get a LR tableaux whose type $\alpha$ has Durfee rank either $m-1$ or $m$, and we are done.
	\end{proof}
	
	\noindent If $f,g\in \Lambda_n$, then we say $f\geq g$ if $f-g$ is Schur-positive. \Cref{durfee_lemma_1} and \Cref{durfee_lemma_2} have the following interesting corollary. 
	
	\begin{corollary}
		\begin{enumerate}
			\item Let $m\geq 2$. Then, for every $\lfloor \frac{n-1}{2}\rfloor\leq k\leq \lfloor \frac{n-1}{2}\rfloor+m^2$, we have $$\displaystyle (\sum_{\substack{\alpha\vdash k\\d(\alpha)=m}}s_{\alpha})(\sum_{\substack{\beta \vdash n-k\\d(\beta)=m}} s_{\beta})\geq \sum_{\substack{\lambda\vdash n \\ d(\lambda)=2m}} s_{\lambda}.$$
			
			\item Let $m\geq 2$. Then, for every $\lfloor \frac{n-1}{2}\rfloor\leq k\leq \lfloor \frac{n-1}{2}\rfloor+m^2-(m-1)$, we have $$\displaystyle (\sum_{\substack{\alpha\vdash k\\m-1\leq d(\alpha)\leq m}}s_{\alpha})(\sum_{\substack{\beta \vdash n-k\\d(\beta)=m}} s_{\beta})\geq \sum_{\substack{\lambda\vdash n \\ d(\lambda)=2m-1}} s_{\lambda}.$$
		\end{enumerate}
	\end{corollary}

	\begin{proof}
		If $k=\lfloor \frac{n-1}{2} \rfloor$, the statements in (1) and (2) are equivalent to \Cref{durfee_lemma_1} and \Cref{durfee_lemma_2}, respectively.
		Let $\lambda \vdash n$ with $d(\lambda)=2m$. Let $H_m$ be the partition obtained from $T_{\lambda}$ by removing the boxes $(i,j)$ where $i\geq m+1$ or $j\geq m+1$. Then $|H_m|\geq \lfloor \frac{n-1}{2} \rfloor+m^2$. Therefore, if $\lfloor \frac{n-1}{2} \rfloor<k\leq \lfloor \frac{n-1}{2} \rfloor+m^2$, the proof of \Cref{durfee_lemma_1} holds without any alteration. This proves (1). 
		
		Now assume $\lambda\vdash n$ with $d(\lambda)=2m-1$. Let $H_m$ be defined exactly as in the previous paragraph. Then $|H_m|\geq \lfloor \frac{n-1}{2} \rfloor+m^2$. If $\beta$ in the proof of \Cref{durfee_lemma_2} is such that $|\beta|\leq \lfloor \frac{n-1}{2} \rfloor+m^2-m+1$, the proof once again works without any alteration.
	\end{proof}
	\begin{lemma}\label{durfee_lemma_3}
		Let $k>0$ and $\alpha\vdash k,\beta \vdash n-k$ be such that $d(\alpha)=d(\beta)=m$. Then there exists $\eta \vdash n$ such that $d(\eta)=m$ and  $c^{\eta}_{\alpha\beta}>0$. The same conclusion is valid if $d(\alpha)=m-1$ and $d(\beta)=m$.
	\end{lemma}
	
	\begin{proof}
		Let $\alpha=(\alpha_1,\alpha_2,\ldots)\vdash k$ and $\beta=(\beta_1,\beta_2,\ldots)\vdash n-k$. We have $\alpha_1\geq\cdots\geq \alpha_m\geq m$ and $\alpha_{m+1}\leq m$. The same holds true for $\beta$. Consider the Young diagram of $\beta$ and fill it with the number $i$ in all the boxes of the $i$-th row. Construct a partition $\eta=(\eta_1,\eta_2,\ldots)\vdash n$ by adjoining the (filled) Young diagram of $\beta$ to $\alpha$ as follows: For $1\leq i\leq m$, the $i$-th row of $T_{\beta}$ is adjoined to the $i$-th row of $T_{\alpha}$. The remaining rows are appended one after the other by putting their boxes one-by-one below the first column, then the second column, and so on until required. Since $\beta_i\leq m$  when $i\geq m+1$, it is not required to go beyond the $m$-th column for appending these rows. By the construction of $\eta$, it easily follows that $\eta\vdash n$ has Durfee rank $m$, $\alpha\subseteq \eta$, and $\eta/\alpha$ is a LR-tableaux of type $\beta$, whence $c^{\eta}_{\alpha\beta}>0$ as required.  The same construction works if we take $d(\alpha)=m$ and $d(\beta)=m-1$, whence $c^{\lambda}_{\beta\alpha}=c^{\lambda}_{\alpha\beta}>0$. We conclude the proof by illustrating the construction with a simple example (see \Cref{Fig_8} below). Let $\alpha=(5,4,3,3,1)\vdash 16$ and $\beta=(6,6,3,3,3,2,1)\vdash 24$. Note that $d(\alpha)=d(\beta)=3$.
		
		\begin{figure}[!h]
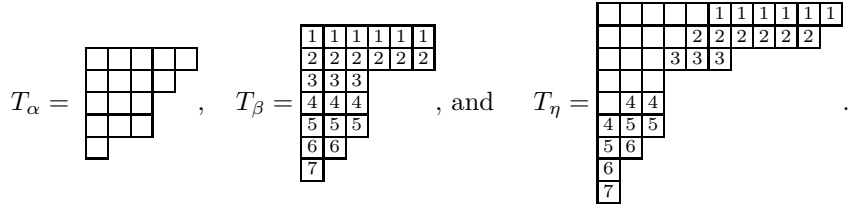

			$T_{\alpha}\;$=\; 
			\begin{ytableau}
			   ~ &   &  &  &\\
			     &   &  &\\
			     &   &\\
			     &   &\\
			     \\
			\end{ytableau},\;\;\;
			$T_{\beta}\;$=\;\begin{ytableau}
				1 & 1 & 1& 1& 1 & 1  \\
				2 & 2 & 2& 2& 2 & 2  \\
				3 & 3 & 3\\
				4 & 4 & 4\\
				5 & 5 & 5\\
				6 & 6\\
				7\\
			\end{ytableau}, and \;\;\;
			$T_{\eta}\;$=\;\begin{ytableau}
				~& & & & & 1 & 1 & 1 & 1 & 1 & 1\\
				 & & & & 2 & 2 & 2& 2& 2 & 2\\
				 & & & 3 & 3 & 3\\
				 &   &\\
				 & 4 & 4\\
				 4 & 5 & 5\\
				 5 & 6\\
				 6\\
				 7
			\end{ytableau}.
		\caption{Construction of $\eta$.}
		\label{Fig_8}
		\end{figure}
	\end{proof}
	
	The following easy observation will be the final ingredient for the proof of \Cref{Theorem_3}(2).
	
	\begin{lemma}\label{durfee_inequality_LR_coeff}
		Let $\mu\vdash m$, $\nu \vdash n$, and $\lambda \vdash m+n$. If $c^{\lambda}_{\mu\nu}>0$, then $d(\lambda)\leq d(\mu)+d(\nu)$.
	\end{lemma}

	\begin{proof}
		Consider the skew shape $T_{\lambda/\mu}$. Clearly, it has a square of  size $d(\lambda)-d(\mu)$. Assume that $d(\lambda)-d(\mu)>d(\nu)$. By assumption, $c^{\lambda}_{\mu\nu}>0$, and hence there exists a LR tableaux of shape $\lambda/\mu$ and type $\nu$. Consider such a filling. In this filling, any number greater than $d(\nu)$ must occur at most $d(\nu)$ times. Since the filling is semi-standard, the left bottom corner of the square must have a filling, say $a$, where $a>d(\nu)$. The last row of the square must be filled with a sequence of numbers that is weakly increasing. Since the reverse reading word of the filling must also be a lattice permutation, if we traverse the row from right to left, we observe that for each entry, we must have filled the diagram with the number $a$ in some box above that row. This accounts for at least $d(\lambda)-d(\mu)$ many occurrences of $a$ in the diagram, a contradiction.
	\end{proof}

	\begin{proof}[\textbf{Proof of \Cref{Theorem_3}(2)}]
		We claim that $c(\chi_{\mu(k)}^2)$ consists of all $\chi_{\lambda}$'s such that $d(\lambda)\leq 2$ except possibly the sign character. But this can be easily verified using \Cref{Kronecker_coeff_two_hook_shapes}. Indeed, if $\lambda=(n-r,1^r)$ where $0\leq r\leq n-2$, the result follows easily by taking $e=f=k$ in \Cref{Kronecker_coeff_two_hook_shapes}. Let $\lambda=(n_4,n_3,d_2,d_1)$ be a double hook as in \Cref{Kronecker_coeff_two_hook_shapes}. Then $n_4+n_3=n-x$ where $x=2d_2+d_1$. Notice that the second summand in \Cref{Kronecker_coeff_two_hook_shapes}(4) is always 1. This yields that $g_{\mu(k)\mu(k)\lambda}>0$, thereby establishing our claim.
		
		Observe that $\chi_{\mu(k)}\in c(\chi_{\mu(k)}^2)$ and hence $c(\chi_{\mu(k)}^i)\subseteq c(\chi_{\mu(k)}^{i+1})$ for $i\geq 2$. Since the sign character $\epsilon$ clearly belongs to $c(\chi_{\mu(k)}^3)$, we conclude that $\epsilon \in c(\chi_{\mu(k)}^i)$ for $i\geq 3$.
		
		\medskip
		
		Now we claim that $c(\chi_{\mu(k)}^i)\setminus \{\epsilon\}=\{\chi_{\lambda}\mid d(\lambda)\leq 2^{i-1}\}\setminus \{\epsilon \}$ for every $2\leq i\leq l$, where $l$ is the least integer such that $2^{l-1}\geq \lfloor \sqrt{n} \rfloor$. Setting $i=l$ in the claim, we obtain that $c(\chi_{\mu(k)}^l)=\irr(S_n)$, and hence our theorem is proved once we establish our claim. We prove it using induction on $i$. For $i=2$, we have the claim from the previous paragraph. Let us assume it is true for all $i$ such that $2\leq i<r\leq l$. We prove the claim for $i=r$. We first show that $\{\chi_{\lambda}\mid d(\lambda)\leq 2^{r-1}\}\setminus \{\epsilon\}\subseteq c(\chi_{\mu(k)}^r)$. Let $(1^n)\neq \lambda \vdash n$ be such that  $2<d(\lambda)\leq 2^{r-1}$.  By using \Cref{durfee_lemma_1}, \Cref{durfee_lemma_2}, and \Cref{durfee_lemma_3}, we conclude that there exists $\alpha\vdash k$, $\beta\vdash n-k$, and $\eta\vdash n$ such that $d(\eta)=\lfloor \frac{d(\lambda)}{2} \rfloor$ or $d(\eta)=\lceil \frac{d(\lambda)}{2} \rceil$ and $c^{\lambda}_{\alpha\beta}c^{\eta}_{\alpha'\beta}>0$, whence \Cref{Kroncker_coeff_and_LR_for_hook} yields that $g_{\lambda\mu(k)\eta}+g_{\lambda\mu(k-1)\eta}>0$. Since $d(\eta)\leq 2^{r-2}$, by induction hypothesis, $\chi_{\eta}\in c(\chi_{\mu(k)}^{r-1})$. Thus, if $g_{\lambda\mu(k)\eta}>0$, then $\chi_{\lambda}\in c(\chi_{\mu(k)}\chi_{\eta})\subseteq c(\chi_{\mu(k)}^r)$ as desired. Otherwise, $g_{\lambda\mu(k-1)\eta}>0$, whence $\chi_{\lambda}\in c(\chi_{\mu(k-1)}\chi_{\eta})\subseteq c(\chi_{\mu(k-1)}\chi_{\mu(k)}^{r-1})\subseteq c(\chi_{\mu(k)}^{r})$ by \Cref{hook_containment_lemma} as desired. Now we show that $c(\chi_{\lambda}^i)\setminus \{\epsilon\}\subseteq \{\chi_{\lambda}\mid d(\lambda)\leq 2^{i-1}\}$. Let $\chi_{\nu}\in c(\chi_{\lambda}^i)\setminus \{\epsilon\}$. By induction hypothesis, we can conclude that $\chi_{\nu}\in c(\chi_{\lambda}\chi_{\mu(k)})$ for some $\lambda$ with $d(\lambda)\leq 2^{i-2}$. This implies that $g_{\lambda\mu(k)\nu}>0$, whence by \Cref{Kroncker_coeff_and_LR_for_hook}, we conclude that there exists $\alpha \vdash k, \beta \vdash n-k$ such that $c^{\lambda}_{\alpha\beta}c^{\nu}_{\alpha'\beta}>0$. Since $d(\lambda)\leq 2^{i-2}$, we obtain that $d(\alpha), d(\beta)\leq 2^{i-2}$. Using \Cref{durfee_inequality_LR_coeff}, we get $d(\nu)\leq 2^{i-1}$ which yields the desired inclusion.
	\end{proof}
	
	\begin{remark}
		Let $\lambda=(\frac{n+1}{2},1^{\frac{n-1}{2}})$ where $n$ is odd.  Then $\ccn(\sigma_{\lambda};S_n)=2$. Thus, in this case, $\ccn(\sigma_{\lambda})$ is much less than $\ccn(\chi_{\lambda})$, unlike the case of two-row partitions. Also note that if $n$ is even, $\ccn(\sigma_{\lambda};S_n)=2$ when $\lambda=(\frac{n}{2},1^{\frac{n}{2}})$, and $\ccn(\sigma_{\lambda};S_n)=3$ when $\lambda=(\frac{n}{2}+1,1^{\frac{n}{2}-1})$.
	\end{remark}
	\section{Proof of \Cref{Theorem_4}}
	 We briefly discuss the irreducible characters of $A_n$ to set down the notations. We denote $\res_{A_n}^{S_n}\chi_{\lambda}$ by $\chi_{\lambda}\downarrow$. Recall that $m(\pi)$ denotes the cycle-type of $\pi\in S_n$ and is a partition of $n$. Let $\text{DOP}(n)$ denote the set of all partitions of $n$ with distinct and odd parts. Further, $\text{SP}(n)$ denotes the set of all self-conjugate partitions of $n$. The folding algorithm defines a bijection $\phi: \mathrm{DOP}(n)\to \mathrm{SP}(n)$ (see \cite[Lemma 4.6.16]{ap}). For $\mu \in \mathrm{DOP}(n)$, the conjugacy class of $S_n$ parameterized by $\mu$ (say $C_{\mu}$) splits into two conjugacy classes of $A_n$ of equal size, that is, $C_{\mu}=C_{\mu}^{+}\sqcup C_{\mu}^{-}$. As convention, we assume that $w_{\mu}\in C_{\mu}^{+}$. Set $w_{\mu}^{+}:=w_{\mu}$  and $w_{\mu}^{-}$ to be a fixed element of $C_{\mu}^{-}$. If $\mu \in \mathrm{DOP}(n)$, we write $\mu=(2m_1+1,2m_2+1,\dots)$. The following theorem describes the irreducible characters of $A_n$ and their character values.
	 
	 \begin{theorem}{\cite[Theorem 5.12.5]{ap}}\label{irreps_alt}
	 	Let $\lambda \vdash n$. Then
	 	\begin{enumerate}
	 		\item If $\lambda\neq \lambda'$, then $\chi_{\lambda}\downarrow$ is an irreducible character of $A_n$. Further, $\chi_{\lambda}\downarrow=\chi_{\lambda'}\downarrow$.
	 		
	 		\item If $\lambda=\lambda'$, then $\chi_{\lambda}\downarrow$ decomposes into two irreducible characters of $A_n$, say $\chi_{\lambda}^{+}$ and $\chi_{\lambda}^{-}$, that is, $\chi_{\lambda}\downarrow=\chi_{\lambda}^{+}+\chi_{\lambda}^{-}$. Moreover, for any odd permutation $\pi$, $\chi_{\lambda}^{-}(w)=\chi_{\lambda}^{+}(\pi w \pi^{-1})$ for all $w \in A_n$.
	 		
	 		\item We have $\chi_{\lambda}^{+}(w)=\chi_\lambda^{-}(w)=\chi_{\lambda}(w)/2$ unless $m(w)$ is the partition $\mu:=\phi^{-1}(\lambda)$ having distinct and odd parts, in which case
	 		$$\chi_{\lambda}^{\pm}(w_{\mu}^{+})=\frac{1}{2}\left(\epsilon_{\mu}\pm \sqrt{\epsilon_{\mu}z_{\mu}}\right),$$
	 		and $\chi_{\lambda}^{\pm}(w_{\mu}^{-})=\chi_{\lambda}^{\mp}(w_{\mu}^{+})$. Here, $\epsilon_{\mu}=(-1)^{m_1+m_2+\cdots}$ and $z_{\mu}$ is the size of the centralizer of $w_{\mu}$ in $S_n$.
	 	\end{enumerate}
	 \end{theorem} 
 	
 	 We begin with an important lemma whose proof is easy.
 	 
 	 \begin{lemma}\label{lemma_alt_1}
 	 	Let $\lambda,\mu\vdash n$ be such that $\lambda\neq \lambda'$ and $\mu=\mu'$. Then $c(
 	 	\chi_{\lambda}\downarrow\chi_{\mu}^{+})\setminus\{\chi_{\mu}^{\pm}\}=c(\chi_{\lambda}\downarrow \chi_{\mu}^{-})\setminus \{\chi_{\mu}^{\pm}\}.$
 	 \end{lemma}
 	
 	\begin{proof}
 		Let $\nu \vdash n$ be such that $\nu\neq \nu'$. Let $\theta=\phi^{-1}(\mu)$. Using \Cref{irreps_alt}, we get
 		\begin{equation}
 			\langle \chi_{\lambda}\downarrow\chi_{\mu}^{+}, \chi_{\nu}\downarrow \rangle=\frac{2}{n!}\left[\sum_{\substack{w\in A_n\\ m(w)\neq \theta}} \frac{\chi_{\lambda}(w)\chi_{\mu}(w)\chi_{\nu}(w)}{2}+\frac{|C_{\theta}|}{2}\chi_{\lambda}(w_{\theta})\chi_{\nu}(w_{\theta})\left(\chi_{\mu}^{+}(w_{\theta}^{+})+\chi_{\mu}^{-}(w_{\theta}^{-})\right)\right].
 		\end{equation}
 	Now computing as above the inner-product $\langle \chi_{\lambda}\downarrow\chi_{\mu}^{-}, \chi_{\nu}\downarrow \rangle$, and using the fact $\chi_{\mu}^{\pm}(w_{\theta}^{-})=\chi_{\mu}^{\mp}(w_{\theta}^{+})$, we conclude that $\langle \chi_{\lambda}\downarrow\chi_{\mu}^{+}, \chi_{\nu}\downarrow \rangle=\langle \chi_{\lambda}\downarrow\chi_{\mu}^{-}, \chi_{\nu}\downarrow \rangle$. If $\nu\vdash n$ is self-conjugate and $\nu\neq \mu$, a similar computation yields $\langle \chi_{\lambda}\downarrow\chi_{\mu}^{+}, \chi_{\nu}^{\pm} \rangle=\langle \chi_{\lambda}\downarrow\chi_{\mu}^{-}, \chi_{\nu}^{\pm}\rangle$.
 	\end{proof}
 	
 	\noindent The following theorem of Bessenrodt and Behns will be required.
 	
 	\begin{theorem}{\cite[Theorem 5.1]{bb}}\label{lemma_alt_2}
 		Let $n\geq 5$ and $\lambda,\mu \vdash n$. Suppose that $d(\chi_{\lambda}\chi_{\mu})=\max\{d(\nu)\mid \chi_{\nu}\in c(\chi_{\lambda})\}.$ Then $d(\chi_{\lambda}\chi_{\mu})=1$ if and only if one of them is $\chi_{(n)}$ or $\chi_{(1^n)}$, and the other one is $\chi_{(n-r,1^r)}$, where $0\leq r\leq n-1$.
 	\end{theorem}
 	
 	\noindent Before moving further, we make the following observation: Suppose $\nu\vdash n$ with $d(\nu)=2$. Then $\nu=\nu'$ implies that $n$ is even. Indeed, if $\nu=\nu'$, then the unfolding $\phi^{-1}(\nu)$ is a partition of $n$ with two distinct and odd parts, whence $n$ is even. As a result, if $n$ is odd, then $\chi_{\nu}\downarrow$ is an irreducible character of $A_n$. 
 	\begin{lemma}\label{lemma_alt_3}
 		Let $n\geq 5$ be odd and $k=\frac{n-1}{2}$. Then, for every $1\leq r\leq \frac{n-3}{2}$, there exists $\nu\vdash n$ with $d(\nu)=2$ and $\chi_{\nu}\downarrow \in c(\chi_{\mu(r)}\downarrow\chi_{\mu(k)}^{+}), c(\chi_{\mu(r)}\downarrow \chi_{\mu(k)}^{-})$.
 	\end{lemma}
 
 	\begin{proof}
 		Using \Cref{Kronecker_coeff_two_hook_shapes} and the previous lemma, we conclude that there exists $\nu\vdash n$ with $d(\nu)=2$ and $\chi_{\nu}\in c(\chi_{\mu(r)}\chi_{\mu(k)})$. This yields that $\chi_{\nu}\downarrow \in c(\chi_{\mu(r)}\downarrow\chi_{\mu(k)}\downarrow)$, whence our assertion holds by \Cref{lemma_alt_1}.
 	\end{proof}
 	\noindent Next, we determine the irreducible constituents of $\chi_{\mu(k)}^{\pm 2}$ and $\chi_{\mu(k)}^{+}\chi_{\mu(k)}^{-}$, where $n$ is odd and $k=\frac{n-1}{2}$.
 	\begin{lemma}\label{lemma_alt_4}
 		Let $n\geq 5$ be odd and $k=\frac{n-1}{2}$. We have the following:
 		\begin{enumerate}
 			\item  If $n\equiv 3(\Mod\; 4)$, then $\displaystyle \chi_{\mu(k)}^{\pm 2}=\sum_{\substack{0\leq i\leq \frac{n-3}{2} \\k\equiv i (\Mod\;2)}}\chi_{\mu(i)}\downarrow + \sum_{\substack{\{\nu,\nu'\} \\ d(\nu)=2}}\chi_{\nu}\downarrow +\chi_{\mu(k)}^{\mp}$.
 			
 			\item If $n\equiv 1(\Mod\; 4)$, then $\displaystyle \chi_{\mu(k)}^{\pm 2}=\sum_{\substack{0\leq i\leq \frac{n-3}{2} \\k\equiv i (\Mod\;2)}}\chi_{\mu(i)}\downarrow + \sum_{\substack{\{\nu,\nu'\} \\d(\nu)=2}}\chi_{\nu}\downarrow +\chi_{\mu(k)}^{\pm}$.
 			
 			\item $\displaystyle \chi_{\mu(k)}^{+}\chi_{\mu(k)}^{-}=\sum_{\substack{0\leq i\leq \frac{n-3}{2} \\k\equiv i+1 (\Mod\;2)}}\chi_{\mu(i)}\downarrow + \sum_{\substack{\{\nu,\nu'\}\\d(\nu)=2}}\chi_{\mu}\downarrow$.
 		\end{enumerate}
 	\end{lemma}
 
 	\begin{proof}
 	  Since we have assumed that $n$ is odd, we conclude that $\nu\neq \nu'$ if $d(\nu)=2$. Using \Cref{Kronecker_coeff_two_hook_shapes}, we have the following decomposition of $\chi_{\mu(k)}\downarrow^2$ into irreducible characters of $A_n$.
 	 \begin{equation}
 	 	\chi_{\mu(k)}\downarrow^2=2\sum_{0\leq i\leq \frac{n-3}{2}}\chi_{\mu(i)}\downarrow +4\sum_{\substack{\{\nu,\nu'\}\\d(\nu)=2}}\chi_{\nu}\downarrow+\chi_{\mu(k)}^{+}+\chi_{\mu(k)}^{-}.
 	 \end{equation}
 	 Notice that $\phi^{-1}(\mu(k))=(n)$. Hence, $\chi_{\lambda}^{\pm}(w_{(n)}^{+})=\frac{1}{2}((-1)^k+\sqrt{(-1)^kn})$. When $0\leq i\leq \frac{n-3}{2}$, a simple inner-product computation using the above fact yields 
 	 
 	 \begin{enumerate}[label=(\alph*)]
 	 	\item $\langle \chi_{\mu(k)}^{\pm 2}, \chi_{\mu(i)}\downarrow \rangle =\frac{1}{4}\langle \chi_{\mu(k)}\downarrow, \chi_{\mu(i)}\downarrow \rangle + \frac{1}{2}(-1)^{k+i}=\frac{1}{2}(1+(-1)^{k+i}).$
 	 	
 	 	\item $\langle \chi_{\mu(k)}^{+}\chi_{\mu(k)}^{-}, \chi_{\mu(i)}\downarrow \rangle =\frac{1}{4}\langle \chi_{\mu(k)}\downarrow, \chi_{\mu(i)}\downarrow \rangle + \frac{1}{2}(-1)^{k+i+1}=\frac{1}{2}(1+(-1)^{k+i+1}).$
 	 	
 	 	\item $\langle \chi_{\mu(k)}^{\pm 2}, \chi_{\nu}\downarrow \rangle=\langle \chi_{\mu(k)}^{+}\chi_{\mu(k)}^{-}, \chi_{\nu}\downarrow \rangle=\frac{1}{4}\langle \chi_{\mu(k)}\downarrow, \chi_{\nu}\downarrow \rangle$ when $d(\nu)=2$.
 	 \end{enumerate}
  	In (a) and (b), we use the fact that $\chi_{\mu(i)}(w_{(n)})=(-1)^i$, and (c) is immediate since $\chi_{\nu}(w_{(n)})=0$. Clearly, (a), (b), and (c) determine the required decomposition except for the membership of $\chi_{\mu(k)}^{+}$ and $\chi_{\mu(k)}^{-}$. It is clear that $\chi_{\mu(k)}^{\pm} \notin c(\chi_{\mu(k)}^+\chi_{\mu(k)}^-)$. Assume that $n\equiv 3(\Mod\;4)$. In this case, we have $\chi_{\mu(k)}^{-}=\overline{\chi_{\mu(k)}^{+}}$ and hence it is easily seen that $\langle \chi_{\mu(k)}^{+2},\chi_{\mu(k)}^{-}\rangle=\langle \chi_{\mu(k)}^{-2},\chi_{\mu(k)}^+ \rangle=1$. When $n\equiv 1(\Mod\;4)$, once again a direct computation yields the desired result.
 	\end{proof}
 	
 	\noindent We are now ready to prove the theorem.

	\begin{proof}[\textbf{Proof of \Cref{Theorem_4}}]
		Since $d(\lambda)=d(\lambda')$, it is clear that $\ccn(\chi_{\mu(k)}\downarrow;A_n)=\ccn(\chi_{\mu(k)};S_n)$. Therefore, when $n$ is even, the result follows by using \Cref{Theorem_3}. Assume that $n$ is odd. Note that $\chi_{\mu(k)}\downarrow^3=\chi_{\mu(k)}^{+3}+\chi_{\mu(k)}^{-3}+\chi_{\mu}^{+2}\chi_{\mu(k)}^{-}+\chi_{\mu(k)}^{+}\chi_{\mu(k)}^{-2}$.  We claim that $c(\chi_{\mu}^{+3})=c(\chi_{\mu(k)}^{-3})\supseteq c(\chi_{\mu}^{+2}\chi_{\mu(k)}^{-})=c(\chi_{\mu(k)}^{+}\chi_{\mu(k)}^{-2})$. We first show that $c(\chi_{\mu(k)}^{+3})\supseteq c(\chi_{\mu(k)}^{+}\chi_{\mu(k)}^{-2})$. Using the previous lemma, we have the following:
		\begin{equation}\label{eqn_1}
			\chi_{\mu(k)}^{+3}=\sum_{\substack{0\leq i\leq \frac{n-3}{2} \\k\equiv i (\Mod\;2)}}(\chi_{\mu(i)}\downarrow\chi_{\mu(k)}^{+}) + \sum_{\substack{\{\nu,\nu'\} \\ d(\nu)=2}}(\chi_{\nu}\downarrow\chi_{\mu(k)}^{+}) +\chi\chi_{\mu(k)}^{+},
		\end{equation}
		where $\chi$ is $\chi_{\mu(k)}^{-}$ (resp. $\chi_{\mu(k)}^{+})$ when $n\equiv 3(\Mod\;4)$ (resp. $n\equiv 1(\Mod\;4)$). Also,
		\begin{equation}\label{eqn_2}
			\chi_{\mu(k)}^{+}\chi_{\mu(k)}^{-2}=\sum_{\substack{0\leq i\leq \frac{n-3}{2} \\k\equiv i (\Mod\;2)}}(\chi_{\mu(i)}\downarrow\chi_{\mu(k)}^{+}) + \sum_{\substack{\{\nu,\nu'\} \\ d(\nu)=2}}(\chi_{\nu}\downarrow\chi_{\mu(k)}^{+}) +\chi'\chi_{\mu(k)}^{+},
		\end{equation}
		where $\chi'$ is $\chi_{\mu(k)}^{+}$ (resp. $\chi_{\mu(k)}^{-})$ when $n\equiv 3(\Mod\;4)$ (resp. $n\equiv 1(\Mod\;4)$). The first two summands of both equations are the same. The last summand in the above two equations differs only by the irreducible characters $\chi_{\mu(i)}\downarrow$ where $0\leq i\leq \frac{n-3}{2}$, and possibly one of $\chi_{\mu(k)}^{+}$ or $\chi_{\mu(k)}^{-}$. By \Cref{lemma_alt_3}, for $1\leq i\leq \frac{n-3}{2}$, we conclude that each $\chi_{\mu(i)}\downarrow$ appear as  constituents of the second summand in both the equations. Further, $\chi_{\mu(k)}^{+}$ and $\chi_{\mu(k)}^{-}$ appears as a constituent of the second summand once again. Finally, note that $\chi\chi_{\mu(k)}^{+}$ contains the trivial character as well, whence the result follows. Now,
		\begin{equation}\label{eqn_3}
			\chi_{\mu(k)}^{-}\chi_{\mu(k)}^{+2}=\sum_{\substack{0\leq i\leq \frac{n-3}{2} \\k\equiv i (\Mod\;2)}}(\chi_{\mu(i)}\downarrow\chi_{\mu(k)}^{-}) + \sum_{\substack{\{\nu,\nu'\} \\ d(\nu)=2}}(\chi_{\nu}\downarrow\chi_{\mu(k)}^{-}) +\rho\chi_{\mu(k)}^{-},
		\end{equation}
		where $\rho$ is $\chi_{\mu(k)}^{-}$ (resp. $\chi_{\mu(k)}^{+})$ when $n\equiv 3(\Mod\;4)$ (resp. $n\equiv 1(\Mod\;4)$). Comparing \Cref{eqn_2} and \Cref{eqn_3}, we observe that the first two summand have the  same irreducible constituents by \Cref{lemma_alt_1}. Also, the third summand in both equations is the same. Finally,
		
		\begin{equation}\label{eqn_4}
			\chi_{\mu(k)}^{-3}=\sum_{\substack{0\leq i\leq \frac{n-3}{2} \\k\equiv i (\Mod\;2)}}(\chi_{\mu(i)}\downarrow\chi_{\mu(k)}^{-}) + \sum_{\substack{\{\nu,\nu'\} \\ d(\nu)=2}}(\chi_{\nu}\downarrow\chi_{\mu(k)}^{-}) +\psi\chi_{\mu(k)}^{-},
		\end{equation}
		where $\psi$ is $\chi_{\mu(k)}^{+}$ (resp. $\chi_{\mu(k)}^{-})$ when $n\equiv 3(\Mod\;4)$ (resp. $n\equiv 1(\Mod\;4)$). Comparing the above equation with \Cref{eqn_1} and by using \Cref{lemma_alt_1}, it follows that $c(\chi_{\mu(k)}^{+3})=c(\chi_{\mu(k)}^{-3})$. Thus, the claim is established, and it follows that $c(\chi_{\mu(k)}^{+3})=c(\chi_{\mu(k)}^{-3})=c(\chi_{\mu(k)}\downarrow)$. Using \Cref{basic_lemma_3}, the theorem follows.
	\end{proof}
	
	\subsection*{Acknowledgment} 
	We thank Jyotirmoy Ganguly for fruitful discussions. We also thank Amritanshu Prasad, Sankaran Viswanath, and Arun Ram for their encouragement and helpful discussions. We thank Mainak Ghosh for his insight in proving \Cref{upperbound_covering_number_sigma_tworow} and  \Cref{covering_number_sigma_tworow_specialcase}. 
	\bibliographystyle{abbrv}
	\bibliography{references}
\end{document}